\documentclass[11pt,a4paper,dvipsnames]{article}

\usepackage{fullpage,amssymb,amsmath,amsfonts,amsthm,graphicx,color,float, mathtools,tikz} 
\usepackage{paralist} 
\usepackage[utf8]{inputenc}
\usepackage[english]{babel}
\usepackage[normalem]{ulem}
\usepackage{xcolor}
\usepackage[textsize=scriptsize, backgroundcolor=blue!20!white,bordercolor=red]{todonotes} 
\usepackage{dsfont} 
\usepackage{mathrsfs}
\usepackage{graphicx}
\usepackage{amsmath}
\usepackage{amsthm}
\usepackage{amsfonts}
\usepackage{amssymb}
\usepackage{multicol}
\usepackage{fancyhdr}
\usepackage{mathtools}

\usepackage{cite}

\usepackage{hyperref} 

\newtheorem{theorem}{Theorem}[section]
\newtheorem{Lemma}[theorem]{Lemma}

\newtheorem{corollary}[theorem]{Corollary}

\theoremstyle{definition}
\newtheorem{definition}[theorem]{Definition}
\newtheorem*{remark*}{Remark}

\newcommand{\R}{\mathbb R}

\newcommand{\N}{\mathbb N}

\newcommand{\norm}[1]{\lVert#1\rVert}

\numberwithin{equation}{section}
\allowdisplaybreaks 

\title{On the mean-field limit for the Vlasov-Poisson system in two dimensions}
\author{Corresponding author: Manuela Feistl-Held\thanks{Fachbereich ANG, Technische Hochschule Rosenheim, Hochschulstraße 1, 83024 Rosenheim, Germany. Email: Manuela.Feistl-Held@th-rosenheim.de}\\  
	Peter Pickl\thanks{Fachbereich Mathematik, Eberhard Karls Universität Tübingen, Auf der Morgenstelle 10, 72076 Tübingen, Germany. Email: p.pickl@uni-tuebingen.de}}

\begin{document}
\maketitle

\begin{abstract}
We present a probabilistic proof of the mean-field limit and propagation of chaos of a classical N-particle system in two dimensions with Coulomb interaction force of the form $f^N(q)=\pm\frac{q}{|q|^2}$ and $N$-dependent cut-off at $|q|>N^{-2}$.
In particular, for typical initial data, we show convergence of the Newtonian trajectories to the characteristics of the Vlasov-Poisson system. The proof is based on a Gronwall estimate for the maximal distance between the exact microscopic dynamics and the approximate mean-field dynamics. Thus our result leads to a derivation of the Vlasov-Poisson equation from the microscopic $N$-particle dynamics with force term arbitrary close to the physically relevant Coulomb force.

\end{abstract}

\section{Introduction} \label{sec:intro}
In this work, we provide a microscopic derivation of the Vlasov–Poisson system in two dimensions, which models the evolution of a plasma consisting of identically charged particles interacting through electrostatic or gravitational forces.
We therefore consider a system of $N$ interacting particles governed by Newtonian dynamics. The state of the system at any given time is represented by a trajectory in phase space $\mathbb{R}^{4N}$, denoted by $X = (Q, P) = (q_1, \ldots, q_N, p_1, \ldots, p_N) \in \mathbb{R}^{4N}$, where each $q_j \in \mathbb{R}^2$ corresponds to the position of the $j$-th particle, and each $p_j \in \mathbb{R}^2$ denotes its momentum.
The evolution of the system is given by the coupled differential equations
\begin{align}
i\in \{1,...,N\},\ \begin{cases}
& {\dot{q}}_i=\frac{1}{m}p_i\\
& \dot{p}_i=\frac{1}{N}\sum_{j\neq i}f^N(q_i-q_j) 
\end{cases} \label{Newtoneq}
\end{align}
with particle mass $m>0$, which will always be set equal to $1$ in our considerations.
We consider a regularized Coulomb interaction in two dimensions with a cut-off at scale $N^{-\beta}$.
\begin{definition}\label{force f}
For $N\in \N\cup\lbrace \infty \rbrace$ the interaction force is given by 
\begin{align*}
f^{N}:\R^{2}\rightarrow\R^{2}, q\mapsto\begin{cases}
aN^{2\beta}q& \text{if}\  |q|\leq N^{-\beta} \\
a \frac{q}{|q|^{2}} & \text{if}\  |q|> N^{-\beta} 
\end{cases}
\end{align*}
for $2\geq\beta>0$.
\end{definition} Note, that we do not have any further constraints on the choice of $a$. In particular we consider both, attractive and repulsive interactions. We  will use the notation $F^N:\R^{4N}\rightarrow\R^{2N}$ for the total force of the system. Thus the j'th compontent of $F^N$ gives the force exhibited on a single coordinate $j$ $$\left(F^N(X)\right)_j\coloneqq  \sum_{i\neq j}\frac{1}{N}f^N(q_i-q_j).$$
We consider the system in the mean-field scaling regime, ensuring that the total mass of the system remains of order one. The prefactor $\frac{1}{N}$ in the interaction term serves as the corresponding scaling factor and is widely adopted in the literature for this purpose (see, e.g., \cite{Jabin lecture notes}).
We derive the Vlasov–Poisson equation from the microscopic Newtonian dynamics of an $N$-particle system.
To this end, we compare the microscopic time evolution $\Psi^N_{t,s}$ of the $N$-particle system with an effective one-particle dynamics governed by the Vlasov flow $(\varphi^N_{t,s})_{t,s \in \mathbb{R}} : \mathbb{R}^2 \to \mathbb{R}^2$, and we establish the convergence of $\Psi^N_{t,s}$ to a product of $\varphi^N_{t,s}$ in the limit $N \to \infty$, in a suitable sense.
Further we obtain weak convergence of the $s$-particle marginals of the $N$-particle system towards the $s$-fold product of solutions to the Vlasov equation, which is commonly referred to as the propagation of molecular chaos.

\subsection{Previous results}
The first mathematical derivation of the Vlasov equation was achieved for systems with Lipschitz-continuous forces. To the best of our knowledge, the first mathematically rigorous derivation dates back to the 1970s and was presented independently by Neunzert and Wick in 1974~\cite{NeunzertWick}, as well as by Braun and Hepp~\cite{BraunHepp,Dobrushin}and Spoon~\cite{SpohnBook}. They proved that if the initial particle distribution is close to a given density $k_0 $, then the particle distribution remains close to the solution $k_t $ of the Vlasov equation at later times.

One of the main challenges in the singular case such as Coulomb or Newtonian interactions is the control of particle clustering, see for instance~\cite{SpohnBook}.
Hauray and Jabin ~\cite{Hauray2007,Hauray2015} were able to treat singular interaction forces scaling like $1/\vert q \vert^{\lambda}$ in three dimensions for ${\lambda} < 1$, and later extended their analysis to the physically more relevant case of $\lambda$ slightly below $2$, under the assumption of a lower bound on the cut-off at $q = N^{-1/6}$  for the full Coulomb case. However, their results remain in the deterministic setting, as they had to prescribe sufficiently generic initial configurations depending on the respective $N$-particle law.
The last deterministic result we want to mention here is~\cite{Kiessling}, which applies to repulsive pair interactions and does not require a cut-off, but instead assumes a bound on the maximal forces in the microscopic system.
Boers and Pickl~\cite{BoersPickl} proposed a novel method for deriving mean-field equations, specifically designed for stochastic initial conditions and interaction forces scaling like $|x|^{-3\lambda + 1}$ with $5/6 < \lambda < 1$. They were able to achieve a cut-off as small as $N^{-\frac{1}{3}}$, which corresponds to the typical interparticle distance. Subsequently, Lazarovici and Pickl~\cite{Dustin} extended the approach from~\cite{BoersPickl} by exploiting the second-order nature of the dynamics and introducing an anisotropic scaling of the underlying metric, allowing them to include the Coulomb singularity. This led to a microscopic derivation of the Vlasov--Poisson equations with a cut-off of $N^{-\delta}$ for some $0 < \delta < \frac{1}{3}$.
More recently, Grass ~\cite{grass} further reduced the cut-off parameter to $N^{-\frac{7}{18}}$ in $d=3$, by refining the analysis of possible collisions and leveraging the second-order structure of the dynamics. This method was then extended by Feistl-Held and Pickl, who succeeded in lowering the cut-off to $N^{-\frac{5}{12}}$.
All of these results were established in the case of spatial dimension $ d = 3$.

In dimension $d=2$, results concerning the derivation of the Vlasov--Fokker--Planck equation were obtained by Bresch, Jabin and Soler in~\cite{BRESCH JABIN SOLER}. They exploited the regularizing effects of Brownian motion to control singular interactions and focused on the repulsive case.

The strategy which we shall present in the following uses stochastic initial conditions. For Coulomb interactions there are in fact deterministic initial conditions for which the dynamics are not described by the Vlasov equation. However, such `bad' initial conditions of particles may, while not impossible, be very atypical in the sense that the respective volume in phase space is small.
More precisely, we prove that the measure of the set where the maximal distance of the Newtonian trajectory and the mean-field trajectory exceeds $N^{-\gamma}$ for $ \gamma>0$ decreases exponentially with the number of particles $N$, as the number of particles grows to infinity.

\subsection{Dynamics of the Newtonian and of the effective system}
By introducing the N-particle force defined in Definition \ref{force f} we can characterize the Newtonian flow $\Psi_{t,s}^N(X)$, which describes the phase space trajectory of $N$ particles, where the first component $\Psi_{t,s}^{1,N}(X)$ corresponds to their positions in physical space, and the second component $\Psi_{t,s}^{2,N}(X)$ corresponds to their velocities. 
\begin{definition}\label{Def:Newtonflow}
The Newtonian flow $\Psi_{t,s}^{N}(X)=\left(\Psi_{t,s}^{1,N}(X)),\Psi_{t,s}^{2,N}(X)\right)$ on $\R^{4N}$ is defined by the solution of
\begin{align}
\frac{d}{dt} \Psi_{t,s}^{N}(X)=\left(\Psi_{t,s}^{2,N}(X),F(\Psi_{t,s}^{1,N}(X))\right) \in \R^{2N}\times\mathbb{R}^{2N} \ \ \ \text{with} \  \Psi_{s,s}^{N}(X)=X. 
\end{align}
\end{definition}
As the vector field is Lipschitz for fixed $N$ we have global existence and uniqueness of solutions and hence a N-particle flow.

The macroscopic law of motion for the particle density is given by a generalisation of the continuity equation, the Vlasov-Poisson equation. This equation describes the time evolution of the distribution function $k:\R^4\rightarrow\R_0^{+}$ of a system of particles that interact through long-range, mean-field forces, like gravity or electrostatics without collisions.
It is given by
\begin{equation}\label{Vlasov}
	\partial_t k + p \cdot \nabla_q k + \Bigl(f^N * \tilde{k}_t \Bigr) \cdot \nabla_p k = 0, 
\end{equation}
with $\tilde{k}_t(q)=\int k_t(q,p)d^2p$ and $N \in \N\cup\lbrace \infty \rbrace$.
Global existence in two space dimensions was established in the work of Okabe and Ukai \cite{ocabe} and Wollman \cite{wollman}, even for singular interactions.
For a fixed initial distribution $k_0 \in L^\infty(\R^2\times \R^2)$ with $k_0 \geq 0$ and $\int k_t(q,p)d^2p=\tilde{k}_t(q)$ we denote by $k^N_t$ the unique solution of the Vlasov-Poisson equation \eqref{Vlasov} with initial datum $k^N_t(0, \cdot,\cdot) = k_0$.\\
The characteristics of Vlasov-Poisson equation similar \eqref{Newtoneq} are given by the following system of Newtonian differential equations
\begin{align}
 \begin{cases}
&\dot{\bar{q}}_i=\bar{p}_i \\
&\dot{\bar{p}}_i=\bar{f}_{t}^{N}:=f^{N}*\tilde{k}_{t}^{N} \label{eff.flow}
\end{cases}
\end{align}
with spatial density $\tilde{k}_t^{N}:\R\times \R^{2}\rightarrow \R_{0}^{+}$ given by 
$\tilde{k}_{t}^{N}(q):=\int k_t^{N}(p,q) d^{2}p$.

The system \eqref{eff.flow} is uniquely solvable on any interval $[0,T]$ and this provides us the flow $(\varphi^{\infty}_{s,t})_{s,t\in \R}$. $(\varphi_{s,t})_{s,t\in \R}=(\varphi^1_{\cdot,s}(x),{\varphi^2_{\cdot,s}(x)})$ solves the equations \eqref{eff.flow} where $\varphi_{s,s}(x)=x$ for any $x\in \R^{2}$ and $s\in \R$.
With this construction we get a new trajectory which is influenced by the mean-field force and not by the pair interaction force like in the Newtonian system.
Now we have two trajectories which we will compare and later show that they are typically close to each other.
To this end, we consider the lift of $\varphi^N_{t,s}(\cdot)$ to the $N$-particle phase-space, which we denote by $\Phi_{t,s}^N$. 
Denoting by $\bar{F}:\R^{2N}\to\R^{2N}$ the lift of the mean field force to the $N$-particle phase-space, i.e. $(\overline{F}_t(Z))_i := f^{N}*\tilde{k}_{t}^{N}(x_i)$ for$ X=(x_1,...,x_N)$ we finally define the mean-field flow analogously to Definition \ref{Def:Newtonflow}.
\begin{definition}\label{Def:Meanfieldflow}
The respective effective flow $\Phi_{t,s}^{N}=\left(\Phi_{t,s}^{1,N},\Phi_{t,s}^{2,N}\right)=\left(\varphi_{t,s}^{N}\right)^{\otimes N}$ satisfies 
\begin{align*}
\frac{d}{dt}\Phi_{t,s}^{N}(X)=\left(\Phi_{t,s}^{2,N}(X),\bar{F}(\Phi_{t,s}^{1,N}(X)\right)
\end{align*}
with $\Phi_{s,s}^{N}(X)=X.$
\end{definition}
Here, $\Phi_{t,s}^{1,N}(X)$ corresponds the positions in physical space, and the second component $\Phi_{t,s}^{2,N}(X)$ corresponds to the velocity.

In contrast to the Newtonian Flow $\Psi_{t,s}^{N}$, the effective flow $\Phi_{t,s}^{N}$ conserves independence, due to the fact that $\Phi^N$ consists of $N$ copies of $\varphi_t$. 

In summary, for fixed $k_0$ and $N \in \N$, we consider for any initial configuration $X \in \R^{4N}$  two different time-evolutions: $\Psi^N_{t,0}(X)$, given by the microscopic equations and $\Phi^N_{t,0}(X)$, given by the time-dependent mean-field force generated by $f^N$. We are going to show that for typical $X$, the two time-evolutions are close in an appropriate sense. 
In other words, we have non-linear time-evolution in which $\varphi^N_{t,s}(\cdot\,; k_0)$ is the one-particle flow induced by the mean-field dynamics with initial distribution $k_0$, while, in turn, $k_0$ is transported with the flow $\varphi^N_{t,s}$. 
\subsection{Statement of the results}
In the following section we show that the $N$-particle trajectory $\Psi_t$ starting from $\Psi_0$ (i.i.d. with the common density $k_0$) remains close to the mean-field trajectory $\Phi_t$ with the same initial configuration $\Psi_0=\Phi_0$ during any finite time $[0,T]$ and so the microscopic and the macroscopic approach describe the same system. 

\begin{theorem} \label{maintheorem}
Let $T>0$, $k_0\in L^1(\R^{4})$ be a probability density fulfilling
$$
\sup_{N\in\N}\sup_{0\leq s\leq T}||\tilde{k}^N_s||_{\infty}\leq\infty
$$ and let $(\Phi^{\infty}_{t,s})_{t, s\in \R}$ be the related lifted effective flow defined in Definition \ref{Def:Meanfieldflow} as well as $({\Psi}^{N}_{t,s})_{t,s\in \R}$ the $N$-particle flow defined in Definition \ref{Def:Newtonflow}. 
 If $\sigma,\varepsilon>0$ and $0<\beta\leq2$, then there exists a $C>0$ such that for all $N\in \N$ it holds that
\begin{align}
\mathbb{P}\left(X\in \R^{6N}:\sup_{0\le s \le T}\left|\Psi_{s,0}^{N}(X)-{\Phi_{s,0}^{\infty}}(X)\right|_{\infty}>  N^{-\frac{2}{5}+2\sigma} \right)\le C N^{-\frac{2}{5}+\varepsilon}.
\end{align}
\end{theorem}
We take advantage of the fact that the potential is logarithmically singular in dimension $d=2$.
Further we analyse the advantages of the second-order nature of the equation in order to prove the theorem. To this end, we transfer more information from the mean-field system to the microscopic particle system, following the approach introduced by Grass in~\cite{grass}. As long as the true and their related mean-field particles are close in phase space, the types of their collisions are expected to be similar. 
Therefore we will divide the particles into sets, a good and a bad set, depending on their mean-field particle partners.
If for certain pairs of particles comparably strong interactions leading to comparably strong correlations are expected, then - depending on their distance and their relative velocity - they will be labelled bad.
For ease of estimations we formulate the notion of being bad  at the hand of the mean-field trajectories. The influence of the true interaction might be estimated wrongly, but, as long as the true trajectory are very close tot the mean-field trajectories, this error will be small.

An advantage is that the number of bad particles is typically much smaller than the total number of particles $N $. If one now weights the contribution to the force term by the probability of a particle being bad, significantly improved estimates can be obtained.

Additionally, by using the integral version of Gronwalls Lemma we will make full use of the second order nature of the dynamics.
If two particle come exceptionally close to each other, one can expect a correspondingly large deviation of the true and mean-field trajectory.
However, for the vast majority, these deviations are typically only of a very limited duration.
In order not to overestimate the deviations between them, it makes sense to compare the dynamics on longer time periods.
The idea of dividing the particles into sets and using the integral version of Gronwalls Lemma were previously implemented in the work of Grass \cite{grass} for two particle sets, a so called good and a bad one in dimension tree.
A heuristic introduction of this technique can be found in \cite{FeistlPickl} Chapter 3.0.3. In this case, of two dimensions, the Bolzmannzylinder such as probability of a hit has to be adjusted and is now given by
\begin{align*}
\mathbb{P}\left(v_{rel}\leq v_{cut} \ \text{and hit}\right)\leq Crv_{cut}^3.
\end{align*}
This heuristically estimated probability will be rigorously derived in the next section.
The proof and notation are based on the work of Grass \cite{grass}.
While his results pertain to the three-dimensional case, some of his estimates can be adapted to our two-dimensional setting.
\subsection{Preliminary studies}
In the following preliminary studies, we provide several essential lemmas and corollaries required for the main results.
Constants appearing in this paper will generically be denoted by $C$.
More precisely we will not distinguish constants appearing in a sequence of estimates, i.e. in an inequality chain $a\le Cb\le Cd$, the constants $C$ may differ.
The following lemmata and corollaries are slight adaptations of the arguments in Grass \cite{grass}, modified to apply in dimension two instead of three.
\begin{Lemma} \label{Lemma distance same order}
Let $T>0$ and $k_0$ be a probability density fulfilling the assumptions of Theorem \ref{maintheorem} where $(\varphi^{N,c}_{t,s})_{t, s\in \R}$ shall be the related effective flow defined in Definition \ref{Def:Meanfieldflow}. Then there exist a $C>0$ such that for all configurations $X,Y\in \R^4,\ N\in \N\cup \{\infty\}$ and $t,t_0\in [0,T]$ it holds that 
\begin{align*}
|\varphi_{t,t_0}^{N}(X)-\varphi_{t,t_0}^{N}(Y)| \le |X-Y|e^{C|t-t_0|}
\end{align*}
and 
\begin{align*}
|f_c^N*\widetilde{k}^{N}_t(^1X)-f^N_c*\widetilde{k}^{N}_t(^1Y)| \le C|^1X-{^1Y}|.
\end{align*}
\end{Lemma}

\begin{proof}
We will use the following short notation $ X=({^1X},{^2X}) $ with ${^1X}\in\mathbb{R}^2 $ and $ {^2X}\in\mathbb{R}^2 $.
For $ N\in \mathbb{N}\cup \{\infty\} $, $ X,Y\in \mathbb{R}^4 $ and $ R>0 $ we define the spherical shell around $ {^1X} $  in $ \mathbb{R}^2 $ by
\begin{align*}
\mu^{R,N}_{X,Y} :=
\begin{cases}
\Bigl\{ Z \in \mathbb{R}^4 \;\Bigl|\; 2\cdot\max\Bigl(N^{-\beta}, \lvert {^1X} - {^1Y} \rvert\Bigr) \le \lvert {^1Z} - {^1X} \rvert \le R \Bigr\} & \text{für } N < \infty, \\[1mm]
\Bigl\{ Z \in \mathbb{R}^4 \;\Bigl|\; 2 \cdot\lvert {^1X} - {^1Y} \rvert \le \lvert {^1Z} - {^1X} \rvert \le R \Bigr\} & \text{für } N = \infty.
\end{cases}
\end{align*}
It holds that
\begin{align}
& \Bigl| f^{N}*\widetilde{k}^N_{t}({^1X}) - f^{N}*\widetilde{k}^N_{t}({^1Y}) \Bigr|\notag\\
& \leq \bigl| \int_{|Z-{^1X}|\le 2\max\bigl(N^{-\beta},|^1X-{^1Y}|\bigr)}
f^{N}({^1X}-Z)-f^{N}({^1Y}-Z)\bigr|\widetilde{k}^N_t(Z)d^2Z \label{lemmaterm1} \\
& +\Biggl|\int_{\mu^{R,N}_{X,Y}}
\Bigl(f^{N}({^1X}-{^1Z})-f^{N}({^1Y}-{^1Z})\Bigr)
{k}^N_t(Z)d^4Z\Biggr| \label{lemmaterm2} \\ 
& + |{^1X}-{^1Y}|\int_{|Z-{^1X}|\ge \max(2|^1X-{^1Y}|,R)}
g^N({^1X}-Z)\widetilde{k}^N_t(Z)d^2Z.\label{lemmaterm3}
\end{align}
For term \eqref{lemmaterm1}, we obtain
\begin{align*}
& \int_{|Z-{^1X}|\le 2\max\bigl(N^{-\beta},|^1X-{^1Y}|\bigr)}
\bigl|f^{N}({^1X}-Z)-f^{N}({^1Y}-Z)\bigr|\widetilde{k}^N_t(Z)d^2Z \\
\le {} & |{^1X}-{^1Y}|\int_{|Z-{^1X}|\le 2N^{-\beta}}
\underbrace{g^{N}(Z-{^1X})}_{\le CN^{2\beta}}\widetilde{k}^N_t(Z)d^2Z\\[1mm]
& + \int_{|Z-{^1X}|\le 2|^1X-{^1Y}|}
\Bigl( \underbrace{|f^{N}({^1X}-Z)|}_{\displaystyle \le |{^1X}-Z|^{-1}}
+|f^{N}({^1Y}-Z)|\Bigr)\widetilde{k}^N_t(Z)d^2Z\\[1mm]
\le {} & C\|\widetilde{k}^N_t\|_{\infty}|{^1Y}-{^1X}|,
\end{align*}
since $
g^N(q)\le C\min\Bigl(N^{2\beta},\frac{1}{|q|^2}\Bigr)$, for all $q\in\mathbb{R}^2.
$
Here the first addend constitutes an upper bound if $ N^{-\beta}\geq |^1X-{^1Y}| $ and the second if the reverse inequality holds.
Further Term \eqref{lemmaterm3} is bounded by \[
C\|\widetilde{k}^N_t\|_{\infty}  \ln^+(\frac{1}{R})|{^1X}-{^1Y}|,
\]
where $ \ln^+(x):=\max\bigl(\ln(x),1\bigr) $ for $ x>0,$ by similar arguments. It remains to estimate Term \ref{lemmaterm2}, which accounts for the contribution of the mass in the annular region of $ \mu^{R,N}_{X,Y} $. Even though this set might be empty if $ R \le 2\max(N^{-\beta},|{^1X}-{^1Y}|) $, we focus on the nontrivial case $ R>2\max(N^{-\beta},|{^1X}-{^1Y}|) $. By the triangle inequality we have
\begin{align*}
& \Biggl|\int_{\mu^{R,N}_{X,Y}}
\Bigl(f^{N}({^1X}-{^1Z})-f^{N}({^1Y}-{^1Z})\Bigr)
{k}^N_t(Z)d^4Z\Biggr|\\[1mm]
\le {} & \Biggl|\int_{\mu^{R,N}_{X,Y}}
\Bigl(f^{N}({^1X}-{^1Z})-f^{N}({^1Y}-{^1Z})\Bigr)
\Bigl(k^N_t(Z)-k^N_t({^1X},{^2Z})\Bigr)d^4Z\Biggr|\\[1mm]
& + \Biggl|\int_{\{2\max(N^{-\beta},|^1X-{^1Y}|)\le |^1Z-{^1X}|\le R\}}
f^{N}({^1Y}-{^1Z})\,\widetilde{k}^N_t({^1X})d^2(^1Z)\Biggr|\\[1mm]
& + \Biggl|\int_{\{2\max(N^{-\beta},|^1X-{^1Y}|)\le |^1Z-{^1X}|\le R\}}
f^{N}({^1X}-{^1Z})\,\widetilde{k}^N_t({^1X})d^2(^1Z)\Biggr|.
\end{align*}
By Newton’s shell theorem (adapted to two dimensions) and the circular symmetry of the integration set, the last two terms cancel because the “mass” (or “charge”) is symmetrically distributed about $ {^1X} $ (or $ {^1Y} $). 

To estimate the first addend, we define
\[
\Delta^N(t):=\sup_{\substack{X,Y\in \mathbb{R}^4 \\ X\neq Y}}\sup_{r,s\in [0,t] }\frac{|\varphi^N_{r,s}(X)-\varphi^N_{r,s}(Y)|}{|X-Y|}.
\]
For the regularized system (and, for short times, also for the non–regularized one) this quantity is well–defined. By means of the mean value theorem (applied to the densities) and the properties of $ g^N $ we obtain
\begin{align}
& \Biggl|\int_{\mu^{R,N}_{X,Y}}
\Bigl(f^{N}({^1X}-{^1Z})-f^{N}({^1Y}-{^1Z})\Bigr)
\Bigl(k^N_t(Z)-k^N_t({^1X},{^2Z})\Bigr)d^4Z\Biggr|\notag\\[1mm]
\le {} & \int_{\mu^{R,N}_{X,Y}}g^{N}({^1X}-{^1Z})\,|{^1X}-{^1Y}|\notag\\[1mm]
&\quad\cdot\Biggl(\sup_{\widetilde{Z}\in \overline{\varphi^N_{0,t}({^1X},{^2Z})\,\varphi^N_{0,t}(Z)}}|\nabla k_0(\widetilde{Z})|\,
\bigl|\varphi^N_{0,t}(Z)-\varphi^N_{0,t}({^1X},{^2Z})\bigr|\Biggr)d^4Z\notag\\[1mm]
\le {} & \Delta^N(t)|{^1X}-{^1Y}|\int_{|^1Z-{^1X}|\le R}\underbrace{g^{N}({^1X}-{^1Z})\,|{^1X}-{^1Z}|}_{\le C|{^1X}-{^1Z}|^{-1}}d^2({^1Z})\notag\\[1mm]
&\quad\cdot\int_{\mathbb{R}^2}\sup_{Z'\in \mathbb{R}^2}\sup_{\widetilde{Z}\in \overline{\varphi^N_{0,t}({^1X},{^2Z})\,\varphi^N_{0,t}(Z',{^2Z})}}|\nabla k_0(\widetilde{Z})|d^2({^2Z})\notag\\[1mm]
\le {} & C\Delta^N(t)R\,|{^1X}-{^1Y}|,
\label{est:grad.dens.4d}
\end{align}
where $ \overline{X_1X_2}:=\{(1-\lambda)X_1+\lambda X_2\in \mathbb{R}^4: \lambda\in [0,1]\} $ for $ X_1,X_2\in \mathbb{R}^4 $. 

In the last step we also used the uniform upper bound on the spatial density  so that
\[
f_{max}:=\sup_{N\in \mathbb{N}}\sup_{0\le s \le T }\|f^N*\widetilde{k}^N_s\|_{\infty}<\infty.
\]
In particular, for any $ Z'\in \mathbb{R}^4 $ with $ |{^2Z'}|\geq 2f_{max}T $ and $ t\in [0,T] $ we have
\[
|{\varphi}^{2,N}_{0,t}(Z')|\geq |{^2Z'}|-f_{max}t\geq \frac{|{^2Z'}|}{2}.
\]
Thus, under the assumption on the decay of $ |\nabla k_0| $ it follows that
\begin{align*}
\int_{\mathbb{R}^2}\sup_{Z'\in \mathbb{R}^2}\sup_{\widetilde{Z}\in \overline{\varphi^N_{0,t}({^1X},{^2Z})\,\varphi^N_{0,t}(Z',{^2Z})}}|\nabla k_0(\widetilde{Z})|d^2({^2Z})\le C.
\end{align*}
Collecting the estimates, we obtain
\begin{align}
\Bigl| f^{N}*\widetilde{k}^N_{t}({^1X}) - f^{N}*\widetilde{k}^N_{t}({^1Y}) \Bigr|
\le C\Bigl(\ln^+\Bigl(\frac{1}{R}\Bigr)+\Delta^N(t)R\Bigr)|{^1X}-{^1Y}|.
\label{est:force}
\end{align}
We now control the growth of $ \Delta^N(t) $. Let $ s,t\in [0,T] $ and $ X,Y\in \mathbb{R}^4 $ with $ X\neq Y $. Choosing
\[
R:=\frac{1}{\Delta^N(t)},
\]
and applying \eqref{est:force} yields, after an argument analogous to the $ \mathbb{R}^4 $ case, the estimate
\begin{align*}
\sup_{s\le r \le t}\Bigl|{\varphi}^{2,N}_{r,s}(X)-{\varphi}^{2,N}_{r,s}(Y)-({^2X}-{^2Y})\Bigr|
\le b(t)e^{\sqrt{C\ln^+(\Delta(t))}(t-s)},
\end{align*}
with
 $
b(t)= C\ln^+(\Delta(t))\Bigl(|{^1X}+{^1Y}|+|{^2X}+{^2Y}|(t-s)\Bigr)(t-s).
$
A similar argument shows that
\begin{align*}
\sup_{s\le r \le t}\Bigl|{\varphi}^{1,N}_{r,s}(X)-{\varphi}^{1,N}_{r,s}(Y)-({^1X}-{^1Y})\Bigr|
\le \Bigl(|{^2X}-{^2Y}|+b(t)e^{\sqrt{C\ln^+(\Delta(t))}(t-s)}\Bigr)(t-s).
\end{align*}
Thus, one obtains for $ X\neq Y $ and $ s,t\in [0,T] $
\begin{align*}
\frac{1}{|X-Y|}\sup_{s\le r \le t}|\varphi^N_{r,s}(X)-\varphi^N_{r,s}(Y)|
&\le 1+ \sup_{s\le r \le t}\frac{|\varphi^N_{r,s}(X)-\varphi^N_{r,s}(Y)-(X-Y)|}{|X-Y|}\\[1mm]
&\le 1+ 2\max_{k\in \{1,2\}}\sup_{s\le r \le t}\frac{|{^k\varphi}^N_{r,s}(X)-{^k\varphi}^N_{r,s}(Y)-({^kX}-{^kY})|}{|X-Y|}\\[1mm]
&\le 1+ C\ln^+(\Delta(t))te^{\sqrt{C\ln^+(\Delta(t))}(t-s)}.
\end{align*}
A corresponding bound for the time–reversed trajectories can be obtained by analogous estimates. Taking the supremum over $ s,t\in [0,t']\subseteq [0,T] $ and $ X\neq Y $, one eventually deduces that for $ t'\in  [0,T] $ (for $ \Delta(t')\ge e $)
\[
\Delta(t')\le 1+ C\ln\bigl(\Delta(t')\bigr)t'e^{\sqrt{C\ln\bigl(\Delta(t')\bigr)}t'}.
\]
This inequality provides an $ N $-independent upper bound for the growth of $ \Delta^N(t) $ and in particular implies that $ \Delta^N(T)<C $. Together with \eqref{est:force}, this yields the Lipschitz continuity of the mean–field force. Consequently, there exists a constant $ C>0 $ such that for arbitrary $ X,Y\in \mathbb{R}^4 $, $ t_0\in [0,T] $ and $ t\in [0,T-t_0] $
\[
|\varphi^N_{t_0+t,t_0}(X)-\varphi^N_{t_0+t,t_0}(Y)|\le C\int_{0}^t|\varphi^N_{s+t_0,t_0}(X)-\varphi^N_{s+t_0,t_0}(Y)|ds+|X-Y|,
\]
and similarly for the time reversed flow. An application of Gronwall’s lemma now completes the proof.
\end{proof}
The following Lemma constitutes the probability of a hit  i.e. the probability of the different types of collisions.
\begin{Lemma}\label{Prob of group}
Let $(\varphi_{t,s}^{N})_{t,s\in\R}$ be the related effective flow $\beta \geq 0$ then there exists an $C>0$ such that for $N^{-a_k}, N^{-b_k}>0, N\in \N$ and $\lbrack t_1,t_2\rbrack \subset \lbrack 0,T\rbrack$ it holds that
\begin{align*}
\mathbb{P}\Big(X \in \R^4:\big( \exists t\in \lbrack t_1,t_2\rbrack: |\varphi_{t,0}^1(X)-\varphi_{t,0}^1(Y)|\leq  N^{-a_{k}}\wedge |\varphi_{t,0}^2(X)-\varphi_{t,0}^2(Y)|
\leq  N^{-b_{k}}\big)\Big)\\
\leq C( N^{-a_{k}} N^{-3b_{k}} (t_2-t_1)+N^{-2a_{k}}\max( N^{-a_{k}}, N^{-b_{k}})^2)
\end{align*}
\end{Lemma}
\begin{proof}
We will use the following short notation $ X=({^1X},{^2X}) $ with ${^1X}\in\mathbb{R}^2 $ and $ {^2X}\in\mathbb{R}^2 $.
If there exists a point in time $t\in \lbrack t_1, t_2\rbrack$ such that the difference in the position $|\varphi_{t,0}^1(X)-\varphi_{t,0}^1(Y)|\leq  N^{-a_{k}}$ and the difference in the velocity is $|\varphi_{t,0}^2(X)-\varphi_{t,0}^2(Y)|
\leq  N^{-b_{k}}$ then it follows by the previous Lemma \ref{Lemma distance same order} that
\begin{align}\label{distance in space MF}
\sup_{0\leq s\leq t}|\varphi_{s,t}(X)-\varphi_{s,t}(Y)|\leq C \max( N^{-a_{k}},N^{-b_{k}}),
\end{align}
 as the distance in space between mean-field particles stays the same order due Lemma \ref{Lemma distance same order}.
If at $t_1$ the difference  $|\varphi_{t_1,0}^1(X)-\varphi_{t_1,0}^1(Y)|\leq  N^{-a_{k}}$ holds, the probability of configurations is given by
\begin{align*}
&\int_{\R^4}\mathds 1_{Z\in\R^4:|^1Z-\varphi_{t_1,0}^1(X)|\leq N^{-a_{k}}\wedge |^2Z-\varphi_{t_1,0}^2(X)|\leq C\max(N^{-a_{k}},N^{-b_{k}})}(X)k_{t_1}^N(X) d^4X\\
&\leq C || k_0||_{\infty}(N^{-2a_{k}} \max (N^{-a_{k}},N^{-b_{k}})^2)\\
&\leq C\Big( N^{-a_{k}} (N^{-b_{k}})^3 (t_2-t_1)+(N^{-a_{k}})^2 \max( N^{-a_{k}}, N^{-b_{k}})^2\Big).
\end{align*} 
If the particles do not fulfil the assumption at the starting point $t_1$ we get by applying inequality \eqref{distance in space MF} that for $M\in \N$ there exists $n\in\lbrace0,\hdots,M-1\rbrace$ such that
\begin{align*}
N^{-a_{k}}\leq |\varphi_{\tilde{t}_n,0}^1(X)-\varphi_{\tilde{t}_n,0}^1(Y)|\leq N^{-a_{k}}+ C \max(N^{-a_{k}},N^{-b_{k}})\frac{t_2-t_1}{M}
\end{align*}
and
\begin{align*}
|\varphi_{\tilde{t}_n,0}^2(X)-\varphi_{\tilde{t}_n,0}^2(Y)|\leq C \max(N^{-a_{k}},N^{-b_{k}}).
\end{align*}
for $\tilde{t}_n:=t_1+\frac{n}{M}(t_2-t_1)$.
And so we get the upper bound for the probability 
\begin{align*}
\sum_{n=0}^{M-1}\int_{\R^4}&\mathds 1_{Z\in\R^2: N^{-a_{k}}\leq|Z-\varphi_{\tilde{t}_n,0}^1(X)|\leq N^{-a_{k}}+C \max(N^{-a_{k}},N^{-b_{k}})\frac{t_2-t_1}{M}}(^1X)\\
\cdot &\mathds 1_{Z\in\R^2: |Z-\varphi_{\tilde{t}_n,0}^2(X)|\leq C \max(N^{-a_{k}},N^{-b_{k}})\frac{t_2-t_1}{M}}(X^2) k_{\tilde{t}_n}^{N}(X) d^4X\\
\leq &C N^{-a_{k}}\Big(\max(N^{-a_{k}},N^{-b_{k}})\Big)^3(t_2-t_1)\\
\leq &C (N^{-4a_{k}}+N^{-a_{k}-3b_{k}})(t_2-t_1),
\end{align*}
which proves the lemma.
\end{proof}
So far all $N$ particles were taken into account as possible interaction partners for the considered particle $X_i$.
This constitutes a worst case estimate.
The possible types of collisions and, accordingly, the impact on the force term can differ.
This will be taken into account later by defining collision classes.

We further introduce the underlying Gronwall Lemma, which takes into account the second order nature of the equation. The unlikely collisions are usually only of a limited duration. An integral Gronwall version can take that into account.
\begin{Lemma}\label{Gronwall Lemma}
Let $u:[0,\infty)\to [0,\infty)$ be a continuous and monotonously increasing map as well as $l,f_1:\R\to [0,\infty)$ and $f_2:\R\times \R\to [0,\infty)$ continuous maps such that for some $n\in \N$ and for all $t_1>0,\ x_1,x_2\geq 0$
\begin{itemize}
\item[(i)]
$\begin{aligned}&
x_1< x_2 \Rightarrow f_2(t_1,x_1)\le f_2(t_1,x_2)
\end{aligned}$
\item[(ii)]
$ \exists K_1,\delta>0: \sup\limits_{\substack{x,y  \in [f_1(0),f_1(0)+\delta]\\s \in [0,\delta]}}|f_2(s,x)-f_2(s,y)|\le K_1|x-y|$.
\item[(iii)]
$ \begin{aligned}
& f_1(t_1)+ \int_0^{t_1}...\int_0^{t_n} f_2(s,u(s))dsdt_n...dt_2<  u(t_1) \ \land \\ 
& f_1(t_1)+\int_0^{t_1}...\int_0^{t_n}  f_2(s,l(s))dsdt_n...dt_2 \geq l(t_1),
\end{aligned}$
\end{itemize} 
then it holds for all $t\geq 0$ that $l(t)\le u(t)$.
\end{Lemma}
The proof of Lemma \ref{Gronwall Lemma} can be found in \cite[Lemma 2.2.1]{grass}.
For the proof of the main result, we require the notion of fluctuation. To this end, we introduce a form of first derivative of $f$, defined as follows
\begin{definition}
\label{force g}
For $N\in \N\cup\lbrace \infty \rbrace$ a function to control $|f(q)-f(q+\xi)|$  is given by 
\begin{align*}
g^{N}:\R^{2}\rightarrow\R^{2}, q\mapsto\begin{cases}
  2N^{2\beta}& \text{if}\ |q|\leq 2 N^{-\beta} \\
 \frac{8}{|q|^{2}} & \text{if}\  |q|> 2N^{-\beta} 
\end{cases} 
\end{align*}
for $0<\beta\leq 2 $.
\end{definition}

Analogously to the total force of the system $F^N$, we  will use the notation $G^N:\R^{4N}\rightarrow\R^{2N}$ the total fluctuation of the system. Thus the j'th component of $G^N$ gives the fluctuation exhibited on a single coordinate $j$: $$(G^N(X))_j\coloneqq  \sum_{i\neq j}\frac{1}{N}g^N(q_i-q_j).$$
Since $f$ and $g$ are not differentiable, we now prove some estimates for differences of function values.
Also an important fact of the system is that the distance between the mean-field particles stay of the same order over time. This is provided by the following Lemma.

\begin{Lemma}\label{LipschitzLemma for f}
\begin{itemize}
\item[a)] For $a,b,c\in \R^2$ with $|a|\leq \min(|b|,|c|)$ the following relations hold
\end{itemize}
\begin{align}\label{minabschätzung}
|f^N(b)-f^N(c)|&\leq g^{N}(a)|b-c|.
\end{align} 
\item[b)]
If $\|X_t-\overline X_t\|_\infty\leq 2N^{-\beta}$, then it holds that
	\begin{align}\label{LipschitzLemma_f}
	\left\| F^N(X_t)-F^N(\overline X_t)\right\|_\infty\leq C\|G^N(\overline X_t)\|_\infty\|X_t-\overline X_t\|_\infty,
	\end{align}
	for some $C>0$ independent of $N$.
\end{Lemma}

     \begin{proof}
     \begin{itemize}
	\item[a)] For the case $|a|\leq 2N^{-\beta}$ we have $\norm{\nabla f^N}_\infty\leq 2 N^{2\beta}$ and thus $2N^{2\beta}$ constitutes a Lipschitz-constant for $f^N.$\\
	For  $|a|\geq 2N^{-\beta},$ we get by the mean value theorem and the fact, that  $\nabla f^N(x)$ is decreasing 
	\begin{align*}
|f^N(b)-f^N(c)|\leq \nabla f^N(a)||b-c|\leq \left(\frac{1}{|a|}\right)^2|b-c|\leq Cg^{N}(a)|b-c|.
	\end{align*}
	\item[b)]
     For any $x,\xi\in\R^2$ with $|\xi|<N^{-\beta}$, we have for $|x|<2N^{-\beta}$
\begin{align}\label{gbound1}
|f^N(x+\xi)-f^N(x)|\leq N^{2\beta}|\xi|\leq g^{N}(x)|\xi|
	\end{align} 
by applying  estimate \ref{minabschätzung} and for choosing without loss of generality $a=b=x+\xi$ and $c=x$.
For $|x|\geq 2N^{-\beta}$ we use the fact that in this case small changes in the argument of the function lead to small changes in the function values, i.e. for  $\xi\leq N^{-\beta}$ we have $g^{N}(x+\xi)\leq C g^{N}(x)$. Thus we have by estimate \ref{minabschätzung}
\begin{align*}
|f^N(x+\xi)-f^N(x)|\leq C g^{N}(x+\xi)|\xi|\leq Cg^{N}(x)|\xi|.
\end{align*}
Applying claim \eqref{gbound1} one has
	\begin{align}
	|(F^N(X_t))_i-(F^N(\overline X_t))_i|&\leq \frac{1}{N}\sum\limits_{j\neq i}^{N}\left| f^N(x_i^t-x_j^t)-f^N(\overline x_i^t-\overline x_j^t)\right|\notag\\
	&\leq \frac{C}{N}\sum\limits_{j\neq i}^{N}g^{N}(\overline x_i^t-\overline x_j^t)\left|x_i^t-x_j^t-\overline x_i^t+\overline x_j^t\right|\notag \\
	&\leq C( g^{N}(\overline X_t))_i\left|X_t-\overline X_t\right|_\infty,
	\end{align}
	which leads to estimate \eqref{LipschitzLemma_f}.
	\end{itemize}
    \end{proof}
Last but not least, we come to the most important corollary of this paper. It provides suitable upper bounds for almost all integrals that arise in the proof of the main theorem. The first two parts of the corollary are used to estimate integrals involving the function $ g$, while the remaining two parts are concerned with estimating integrals of $ f$.

\begin{corollary}\label{corollary phi and  psi}
Let $k_0$ be a probability density fulfilling the assumptions of Theorem \ref{maintheorem} and $(\varphi^{N}_{t,s})_{t, s\in \R}$ be the related effective flow defined in \eqref{Def:Meanfieldflow} as well as $(\Psi^{N}_{t,s})_{t,s\in \R}$ the $N$-particle flow defined in \eqref{Def:Newtonflow}. Let additionally for $N,n\in \N$, $C>0$ and $c_N>0$ $h_N, l_N:\R^{2}\to \R^n$ be  continuous maps fulfilling $$ |h_{N}(q)|\le\begin{cases} C c_N^{-2},&\ |q|\le c_N\\\frac{C}{|q|^{2}},&\   |q|> c_N \end{cases},$$
 $$ |l_N(q)|\le\begin{cases} C c_N^{-1},&\ |q|\le c_N\\\frac{C}{|q|},&\   |q|> c_N \end{cases}.$$
\begin{itemize}
\item[(i)] Let for $Y,Z\in \R^4$ $t_{min}\in [0,T]$ be a point in time where 
\begin{align*}
 \min_{0\le s \le T}|\varphi^{1,N}_{s,0}(Z)-{^1\varphi^{N,c}_{s,0}}(Y)|=&|\varphi^{1,N}_{t_{min},0}(Z)-{\varphi^{1,N}_{t_{min},0}}(Y)|=:\Delta r>0 \ \land \\
 &|\varphi^{2,N}_{t_{min},0}(Z)-{\varphi^{2,N}_{t_{min},0}}(Y)|=:\Delta v>0,
\end{align*}
then there exists $C>0$ (independent of $Y,Z\in \R^{2}$ and $N\in \mathbb{N}$) such that
\begin{itemize}
\item[(i1)]
\begin{align*}
& \int^{T}_{0}|h_N(\varphi^{1,N}_{s,0}(Z)-{\varphi^{1,N}_{s,0}}(Y))|ds
\le  C \min\big(\frac{1}{\Delta r^{2}},\frac{1}{c_N\Delta v},\frac{1}{\Delta r\Delta v}\big).
\end{align*}
\item[(i2)]
\begin{align*}
& \int^{T}_{0}|l_N(\varphi^{1,N}_{s,0}(Z)-{\varphi^{1,N}_{s,0}}(Y))|ds
\le  C \min\big(\frac{1}{\Delta r},\frac{1}{c_N},\frac{\ln(\frac{\Delta v}{\Delta r})}{\Delta v},\frac{1}{\Delta v}\big).\\
\end{align*}
\end{itemize}
\item[(ii)] Let $T>0$, $i,j\in \{1,...,N\}$, $i\neq j$, $X\in \R^{4N}$ and $Y,Z\in \R^4$ be given such that for some $\delta>0$
$$N^{\delta}|\varphi^{1,N}_{t_{min},0}(Y)-{\varphi^{1,N}_{t_{min},0}}(Z)|\le |\varphi^{2,N}_{t_{min},0}(Y)-{\varphi^{2,N}_{t_{min},0}}(Z)|=:\Delta v$$
and
\[\sup_{0\le s \le T}|\varphi^{N}_{s,0}(Y)-[{\Psi^{N}_{s,0}}(X)]_i|\le N^{-\delta}\Delta v\land \sup_{0\le s \le T}|\varphi^{N}_{s,0}(Z)-[{\Psi^{N}_{s,0}}(X)]_j|\le N^{-\delta}\Delta v\]
where $t_{min}$ shall fulfil the same conditions as in item (i). Then there exist $N_0\in \N$ and $C>0$ (independent of $X\in \R^{4N}$, $Y,Z\in \R^4$) such that for all $N\geq N_0$ 
\begin{itemize}
\item[(ii1)]
\begin{align*}
&\int^{T}_{0}|h_N([\Psi^{1,N}_{s,0}(X)]_i-[^1\Psi^{N}_{s,0}(X)]_j)|ds\\
\le & C\min\big(\frac{1}{c_N\Delta v}, \frac{1}{\min\limits_{0\le s\le T}|[\Psi^{1,N}_{s,0}(X)]_i-[\Psi^{1,N}_{s,0}(X)]_j|\Delta v}\big).
\end{align*}
\item[(ii2)]
\begin{align*}
&\int^{T}_{0}|l_N([\Psi^{1,N}_{s,0}(X)]_i-[\Psi^{1,N}_{s,0}(X)]_j)|ds\\
\le & C\min\left(\frac{1}{\Delta v},\frac{\ln\left(\min\limits_{0\le s\le T}|[\Psi^{1,N}_{s,0}(X)]_i-[\Psi^{1,N}_{s,0}(X)]_j|\right)}{\Delta v}\right).
\end{align*}
\end{itemize}
\end{itemize}
\end{corollary}
\begin{proof}
\begin{itemize}
\item[(i)]
In a first step we want to derive an appropriate upper bound for the relative velocity between (mean-field) particles at times when they are `close' to each other. It will turn out by application of Lemma \ref{Lemma distance same order} that the variables $\Delta r$ and $\Delta v$ which we introduced in the assumptions of the Corollary are sufficient to determine such a bound. To this end, we remark that according to Lemma \ref{Lemma distance same order} there exists a constant $C\geq 1 $ such that for all $t\in [0,T]$ and $N\in \mathbb{N}$
\begin{align}
|\varphi^{N}_{t,0}(Z)-\varphi^{N}_{t,0}(Y)|\le C\min_{0\le s \le T}|\varphi^{N}_{s,0}(Z)-\varphi^{N}_{s,0}(Y)| . \label{cons.0}
\end{align}
Thus, it holds for arbitrary $t_1,t_2\in [0,T]$ that the condition 
\begin{align*}
&|\varphi^{1,N}_{t_1,0}(Z)-{\varphi^{1,N}_{t_1,0}}(Y)|\le  |\varphi^{2,N}_{t_1,0}(Z)-{\varphi^{2,N}_{t_1,0}}(Y)| 
\end{align*}
implies
\begin{align*}
|\varphi^{N}_{t_2,0}(Z)-{\varphi^{N}_{t_2,0}}(Y)|\le C|\varphi^{N}_{t_1,0}(Z)-{\varphi^{N}_{t_1,0}}(Y)|\le 2C |\varphi^{2,N}_{t_1,0}(Z)-{\varphi^{2,N}_{t_1,0}}(Y)| . 
\end{align*}
Hence, in any case it holds that 
\begin{align*}
&\max\big(|\varphi^{1,N}_{t_1,0}(Z)-{\varphi^{1,N}_{t_1,0}}(Y)|,|\varphi^{2,N}_{t_1,0}(Z)-{\varphi^{2,N}_{t_1,0}}(Y)|\big)\\
\geq  &\max\big(|\varphi^{1,N}_{t_1,0}(Z)-{\varphi^{1,N}_{t_1,0}(Y)|,\frac{1}{2C}|\varphi^{2,N}_{t_2,0}}(Z)-{\varphi^{2,N}_{t_2,0}}(Y)|\big) \\
\geq &\ \frac{1}{ 2C }  \max\big(\min_{0\le s \le T}|\varphi^{1,N}_{s,0}(Z)-{\varphi^{1,N}_{s,0}}(Y)|,|\varphi^{2,N}_{t_2,0}(Z)-{\varphi^{2,N}_{t_2,0}}(Y)|\big). 
\end{align*}
Let for $Y,Z\in \mathbb{R}^4$ $t_{min}\in [0,T]$ be a point in time where 
$$\min_{0\le s \le T}|\varphi^{1,N}_{s,0}(Z)-{\varphi^{1,N}_{s,0}}(Y)|=|\varphi^{1,N}_{t_{min},0}(Z)-{\varphi^{1,N}_{t_{min},0}}(Y)|=:\Delta r $$
as well as
 $$|\varphi^{2,N}_{t_{min},0}(Z)-{\varphi^{2,N}_{t_{min},0}}(Y)|=:\Delta v,$$
then the previous considerations (applied for $t_2=t_{min}$) yield that for any $t_1\in [0,T]$ the relation
\begin{align}
& \max\big(|\varphi^{1,N}_{t_1,0}(Z)-{\varphi^{1,N}_{t_1,0}}(Y)|,|\varphi^{2,N}_{t_1,0}(Z)-{\varphi^{2,N}_{t_1,0}}(Y)|\big)  \geq  \frac{1}{ 2C }  \max\big(\Delta r, \Delta v\big) \label{max phi dist}
\end{align}
is fulfilled which will be important shortly.\\  According to \cite{grass}[Lemma 2.13] there exists $C>0$ (independent of $X,X'\in \mathbb{R}^4$ and $N$) such that for arbitrary $0\le t_0,t \le T$ where $|t-t_0|\le 1$ 
 \begin{align}
 & \big|{\varphi^{2,N}_{t,t_0}}(X)-{^2\varphi^N_{t,t_0}}(X')-({^2X}-{^2X'})\big| \notag\\
\le & C|t-t_0|\big(|{^1X }-{^1X'}|+|{^2X }-{^2X'}|  |t-t_0|\big) \label{lem3.free.eq.}.
\end{align} 
Let for $t\in [0,T]$ and $C\geq 0$ $t'_{min}\in [t,\min\big(t+\frac{1}{C},T\big)]=:I_{t}$ denote (one of) the point(s) in time where 
$$\min_{s\in I_{t}}|\varphi^{1,N}_{s,0}(Z)-{\varphi^{1,N}_{s,0}}(Y)|=|\varphi^{1,N}_{t'_{min},0}(Z)-{\varphi^{1,N}_{t'_{min},0}}(Y)| $$ and for notational convenience we abbreviate
$$\widetilde{Z}:=\varphi^{N}_{t'_{min},0}(Z) \text{ and }\widetilde{Y}:=\varphi^{N}_{t'_{min},0}(Y).$$
If we choose $C=\lceil 2\sqrt{C} \rceil$ and regard the choice of $t'_{min}$ (in the third step), then relation \eqref{lem3.free.eq.} (applied for $t_0=t'_{min}$) yields that for $s\in I_t$
\begin{align*}
& |\varphi^{1,N}_{s,t'_{min}}(\widetilde{Z})-{\varphi^{1,N}_{s,t'_{min}}}(\widetilde{Y})| \\
\geq & \big| ({^1\widetilde{Z}}-{^1\widetilde{Y}})+({^2\widetilde{Z}}-{^2\widetilde{Y}})(s-t'_{min})\big|\\
& -\big|\int_{t'_{min}}^s\big( ^2\varphi^{N}_{r,t'_{min}}(\widetilde{Z})-{^2\varphi^{N}_{r,t'_{min}}}(\widetilde{Y})\big)-({^2\widetilde{Z}}-{^2\widetilde{Y}})dr\big| \\
\geq &\big|({^1\widetilde{Z}}-{^1\widetilde{Y}})+({^2\widetilde{Z}}-{^2\widetilde{Y}})(s-t'_{min})\big|\\
& - C\underbrace{|s-t'_{min}|^2}_{\le( \frac{1}{C})^2\le \frac{1}{4C}}\big(|{^1\widetilde{Z}}-{^1\widetilde{Y}}|+|{^2\widetilde{Z}}-{^2\widetilde{Y}}||s-t'_{min}|\big)\\
\geq & \max\big(|{^1\widetilde{Z}}-{^1\widetilde{Y}}|,|{^2\widetilde{Z}}-{^2\widetilde{Y}}||s-t'_{min}| \big)-\frac{1}{4}(|{^1\widetilde{Z}}-{^1\widetilde{Y}}|+|{^2\widetilde{Z}}-{^2\widetilde{Y}}||s-t'_{min}|)\\
\geq & \frac{1}{2} \max\big(|{^1\widetilde{Z}}-{^1\widetilde{Y}}|,|{^2\widetilde{Z}}-{^2\widetilde{Y}}||s-t'_{min}| \big)
\end{align*}
which implies that
\begin{align}
&\int^{\min(t+\frac{1}{C},T)}_{t}|l_N(\varphi^{1,N}_{s,0}(Z)-{\varphi^{1,N}_{s,0}}(Y))|ds \notag \\
\le &C\int^{\frac{1}{C}}_{0}\min\Big(\frac{1}{\max\big(|^1\widetilde{Z}-{^1\widetilde{Y}}|,|^2\widetilde{Z}-{^2\widetilde{Y}}|s\big)^{1}},c_N^{-1}\Big)ds \notag \\
\le &C\min\left(\frac{1}{|^2\widetilde{Z}-{^2\widetilde{Y}}|},\frac{\ln\left(\frac{|^2\widetilde{Z}-{^2\widetilde{Y}}|}{|^1\widetilde{Z}-{^1\widetilde{Y}}|}\right)}{|^2\widetilde{Z}-{^2\widetilde{Y}}|}\right) \label{coll.est.1}
\end{align}
and 
\begin{align}
&\int^{\min(t+\frac{1}{C},T)}_{t}|h_N(\varphi^{1,N}_{s,0}(Z)-{\varphi^{1,N}_{s,0}}(Y))|ds \notag \\
\le &C\int^{\frac{1}{C}}_{0}\min\Big(\frac{1}{\max\big(|^1\widetilde{Z}-{^1\widetilde{Y}}|,|^2\widetilde{Z}-{^2\widetilde{Y}}|s\big)^{2}},c_N^{-2}\Big)ds \notag \\
\le &C\min\Big(\frac{1}{c_N|^2\widetilde{Z}-{^2\widetilde{Y}}|},\frac{1}{|^1\widetilde{Z}-{^1\widetilde{Y}}||^2\widetilde{Z}-{^2\widetilde{Y}}|}\Big) \label{coll.est.1}
\end{align}
where we used the properties of $l_N$ and $h_N$ stated in the assumptions of the Corollary. \\
Moreover, the constraints on $l_N$ and $h_N$ directly imply that
$$\int^{t+\frac{1}{C}}_{t}|l_N(\varphi^{1,N}_{s,0}(Z)-{\varphi^{1,N}_{s,0}}(Y))|ds\le \frac{1}{C}\min\big(\frac{1}{|^1\widetilde{Z}-{^1\widetilde{Y}}|},\frac{1}{c_N}\big).$$ and
$$\int^{t+\frac{1}{C}}_{t}|h_N(\varphi^{1,N}_{s,0}(Z)-{\varphi^{1,N}_{s,0}}(Y))|ds\le \frac{1}{C}\min\big(\frac{1}{|^1\widetilde{Z}-{^1\widetilde{Y}}|^2},\frac{1}{c_N^2}\big).$$
After merging the upper bounds it follows that 
\begin{align}
& \int^{\min(t+\frac{1}{C},T)}_{t}|l_N(\varphi^{1,N}_{s,0}(Z)-{\varphi^{1,N}_{s,0}}(Y))|ds \notag \\
\le & C\min\left(\frac{1}{|^1\widetilde{Z}-{^1\widetilde{Y}}|},\frac{1}{c_N},\frac{1}{|^2\widetilde{Z}-{^2\widetilde{Y}}|},\frac{\ln\left(\frac{|^2\widetilde{Z}-{^2\widetilde{Y}}|}{|^1\widetilde{Z}-{^1\widetilde{Y}}|}\right)}{|^2\widetilde{Z}-{^2\widetilde{Y}}|}\right). \label{l bound cor}
\end{align}
\begin{align}
& \int^{\min(t+\frac{1}{C},T)}_{t}|h_N(\varphi^{1,N}_{s,0}(Z)-{\varphi^{1,N}_{s,0}}(Y))|ds \notag \\
\le & C\min\big(\frac{1}{|^1\widetilde{Z}-{^1\widetilde{Y}}|^2},\frac{1}{c_N^2},\frac{1}{c_N|^2\widetilde{Z}-{^2\widetilde{Y}}|},\frac{1}{|^1\widetilde{Z}-{^1\widetilde{Y}}||^2\widetilde{Z}-{^2\widetilde{Y}}|}\big), \label{l bound cor2}
\end{align}
for arbitrary $t$ such that the short interval $[t,\min(t+\frac{1}{C},T)]$ belongs to $[0,T]$.
Further $\frac{1}{C}$ can be chosen independent of $N$ and the considered configurations and according to relation \eqref{max phi dist}
and $|^1\widetilde{Z}-{^1\widetilde{Y}}|\geq\Delta r$ we obtain
{\allowdisplaybreaks
\begin{align*}
& \min\left(\frac{1}{|^1\widetilde{Z}-{^1\widetilde{Y}}|},\frac{1}{c_N},\frac{1}{|^2\widetilde{Z}-{^2\widetilde{Y}}|},\frac{\ln\left(\frac{|^2\widetilde{Z}-{^2\widetilde{Y}}|}{|^1\widetilde{Z}-{^1\widetilde{Y}}|}\right)}{|^2\widetilde{Z}-{^2\widetilde{Y}}|}\right)\\
\leq& C \min\left(\frac{1}{\Delta r},\frac{1}{c_N},\frac{1}{\Delta v},\frac{\ln(\frac{\Delta v}{\Delta r})}{\Delta v}\right)
\end{align*} }
and
{\allowdisplaybreaks
\begin{align*}
& \min\big(\frac{1}{|^1\widetilde{Z}-{^1\widetilde{Y}}|^2},\frac{1}{c_N^2},\frac{1}{c_N|^2\widetilde{Z}-{^2\widetilde{Y}}|},\frac{1}{|^1\widetilde{Z}-{^1\widetilde{Y}}||^2\widetilde{Z}-{^2\widetilde{Y}}|}\big)\\
\leq& C\min\big(\frac{1}{\Delta r^2},\frac{1}{c_N^2},\frac{1}{c_N\Delta v},\frac{1}{\Delta r\Delta v}\big).
\end{align*} }

Thus, estimates \eqref{l bound cor} and\eqref{l bound cor2} imply that 
\begin{align*}
& \int^{T}_{0}|l_N(\varphi^{1,N}_{s,0}(Z)-{\varphi^{1,N}_{s,0}}(Y))|ds\\
\le &\big\lceil\frac{T}{\frac{1}{C}}\big\rceil \sup_{0\le t\le T}\int^{\min(t+\frac{1}{C},T)}_{t}|l_N(\varphi^{1,N}_{s,0}(Z)-{\varphi^{1,N}_{s,0}}(Y))|ds\\
\le &C \min\left(\frac{1}{\Delta r},\frac{1}{c_N},\frac{1}{\Delta v},\frac{\ln(\frac{\Delta v}{\Delta r})}{\Delta v}\right).
\end{align*}
and
\begin{align*}
& \int^{T}_{0}|h_N(\varphi^{1,N}_{s,0}(Z)-{\varphi^{1,N}_{s,0}}(Y))|ds\\
\le &\big\lceil\frac{T}{\frac{1}{C}}\big\rceil \sup_{0\le t\le T}\int^{\min(t+\frac{1}{C},T)}_{t}|h_N(\varphi^{1,N}_{s,0}(Z)-{\varphi^{1,N}_{s,0}}(Y))|ds\\
\le & C\min\big(\frac{1}{\Delta r^2},\frac{1}{c_N^2},\frac{1}{c_N\Delta v},\frac{1}{\Delta r\Delta v}\big).
\end{align*}

\item[(ii)]
Let $X=(X_1,...,X_N)\in \mathbb{R}^{4N}$, $Y_i,Y_j\in \mathbb{R}^4$ and $\Delta v>0$ such that 
\begin{align}
\max_{k\in \{i,j\}}\sup_{0\le t \le T}|\varphi^{N}_{t,0}(Y_k)-[{\Psi}^{N}_{t,0}(X)]_k|\le N^{-\delta}\Delta v \label{Gleichung Cor. 0}
\end{align}
for some $\delta>0$ as well as  
\begin{align}
N^{\delta}|\varphi^{1,N}_{t_{min},0}(Y_i)-{\varphi^{1,N}_{t_{min},0}}(Y_j)|\le |\varphi^{2,N}_{t_{min},0}(Y_i)-{\varphi^{2,N}_{t_{min},0}}(Y_j)|= \Delta v. \label{Gleichung Cor. 1}
\end{align} 
where as usual $t_{min}$ shall again denote a point in time where $|\varphi^{1,N}_{\cdot ,0}(Y_i)-{\varphi^{1,N}_{\cdot,0}}(Y_j)|$ attains its minimal value on $[0,T]$. It follows by Lemma \ref{Lemma distance same order} that there exists a constant $C>0$ such that
\begin{align}
&\frac{\Delta v}{C}\le \min_{0\le t \le T}|\varphi^N_{t,0}(Y_i)-{\varphi^N_{t,0}}(Y_j)|
\le \max_{0\le t \le T}|\varphi^N_{t,0}(Y_i)-{\varphi^N_{t,0}}(Y_j)|
\le  C\Delta v  \label{Gleichung Cor. 1,5}
\end{align}
and for large enough $N\in \mathbb{N}$ it holds according to \eqref{Gleichung Cor. 0} that
\begin{align}
\min_{0\le t \le T}|[\Psi^{N}_{t,0}(X)]_i-[\Psi^{N}_{t,0}(X)]_j |
\geq   \min_{0\le t \le T}|\varphi^{N}_{t,0}(Y_i)-{\varphi^{N}_{t,0}}(Y_j)|-2N^{-\delta}\Delta v 
\geq  \frac{\Delta v}{2C}. \label{Gleichung Cor. 2}
\end{align}
Let for $t'\in [0,T]$ and $C>1$ $t'_{min}$ denote a point in time where $|[{\Psi}^{1,N}_{.,0}(X)]_i-[{\Psi}^{1,N}_{.,0}(X)]_j|$ attains its minimum on $[t',t'+\frac{1}{C}]$. For a compact notation we identify $\widetilde{X}_i:=[\Psi^{N}_{t'_{min},0}(X)]_i$ and $\widetilde{X}_j:=[\Psi^{N}_{t'_{min},0}(X)]_j$.\\
If $|^1\widetilde{X}_i-{^1\widetilde{X}_j}|\geq \frac{\Delta v}{4C}$, then the properties of $l_N$ and $h_N $ yield that 
\begin{align*}
&\int^{t'_{min}+\frac{1}{C}}_{t'_{min}}|l_N([\Psi^{1,N}_{s,0}(X)]_i-[\Psi^{1,N}_{s,0}(X)]_j)|ds
\le  C\min\big(\frac{1}{c_N}, \frac{1}{|^1\widetilde{X}_i-{^1\widetilde{X}_j}|}\big)\le  \frac{C}{\Delta v}
\end{align*} 
and 
\begin{align*}
\int^{t'_{min}+\frac{1}{C}}_{t'_{min}}|h_N([\Psi^{1,N}_{s,0}(X)]_i-[\Psi^{1,N}_{s,0}(X)]_j)|ds
\le&  C\min\big(\frac{1}{c_N^2}, \frac{1}{|^1\widetilde{X}_i-{^1\widetilde{X}_j}|^2}\big)\\
\le& \frac{C}{\Delta v}\min\big(\frac{1}{c_N}, \frac{1}{|^1\widetilde{X}_i-{^1\widetilde{X}_j}|}\big).
\end{align*}
For $|^1\widetilde{X}_i-{^1\widetilde{X}_j}|< \frac{\Delta v}{4C}$ it holds, due to  \eqref{Gleichung Cor. 2}, that
\begin{align} 4C|^1\widetilde{X}_i-{^1\widetilde{X}_j}|\le \Delta v \le 4C|^2\widetilde{X}_i-{^2\widetilde{X}_j}|. \label{Gleichung Cor. 3}  \end{align} 
Moreover, according to Lemma \cite{grass}[Lemma 2.1.3] it holds for all $|s-t_{min}|\le \frac{1}{C}$ that
\begin{align}
&\big|\varphi^{2,N}_{s,0}(Y_i)-{\varphi^{2,N}_{s,0}}(Y_j)- (\varphi^{2,N}_{t'_{min},0}(Y_i)-{\varphi^{2,N}_{t'_{min},0}}(Y_j)) \big| \notag \\
\le &C |s-t'_{min}|\Big(|\varphi^{1,N}_{t'_{min},0}(Y_i)-{\varphi^{1,N}_{t'_{min},0}}(Y_j)|  +|\varphi^{2,N}_{t'_{min},0}(Y_i)-{\varphi^{2,N}_{t'_{min},0}}(Y_j)| |s-t'_{min}|\Big) \notag
\end{align}
which by application of relations \eqref{Gleichung Cor. 0},\eqref{Gleichung Cor. 1,5} and \eqref{Gleichung Cor. 3} yields that
\begin{align}
&\big|\varphi^{2,N}_{s,0}(Y_i)-{\varphi^{2,N}_{s,0}}(Y_j)- (\varphi^{2,N}_{t'_{min},0}(Y_i)-{\varphi^{2,N}_{t'_{min},0}}(Y_j)) \big| \notag \\
\le &C |s-t'_{min}|\Big(\big(\underbrace{|^1\widetilde{X}_i-{^1\widetilde{X}_j}|}_{\le |^2\widetilde{X}_i-{^2\widetilde{X}_j}|}+2N^{-\delta}\Delta v\big)  +C\Delta v |s-t'_{min}|\Big) \notag 
\le  C|s-t'_{min}| |^2\widetilde{X}_i-{^2\widetilde{X}_j}|
\end{align}
According to this estimate and \eqref{Gleichung Cor. 0} it follows for sufficiently large values of $N\in \mathbb{N}$ and $|t-t'_{min}|\le \frac{1}{C}$ that

\begin{align*}
&\big|([\Psi^{1,N}_{t,0}(X)]_i-[\Psi^{1,N}_{t,0}(X)]_j) -(^1\widetilde{X}_i-{^1\widetilde{X}_j})-(^2\widetilde{X}_i-{^2\widetilde{X}_j})(t-t'_{min})\big| \notag \\
= & \big|\int_{t'_{min}}^{t}([\Psi^{2,N}_{s,0}(X)]_i-[\Psi^{2,N}_{s,0}(X)]_j)ds-(^2\widetilde{X}_i-{^2\widetilde{X}_j})(t-t'_{min}) \big| \notag \\
\le & \big|\int_{t'_{min}}^{t}([\Psi^{2,N}_{s,0}(X)]_i-[\Psi^{2,N}_{s,0}(X)]_j)-(\varphi^{2,N}_{s,0}(Y_i)-{\varphi^{2,N}_{s,0}}(Y_j))ds \big| \notag \\
& +\big|\int_{t'_{min}}^t(\varphi^{2,N}_{s,0}(Y_i)-{\varphi^{2,N}_{s,0}}(Y_j))- (\varphi^{2,N}_{t'_{min},0}(Y_i)-{\varphi^{2,N}_{t'_{min},0}}(Y_j))ds \big| \notag \\
& + \big|({\varphi^{2,N}_{t'_{min},0}}(Y_i)-{\varphi^{2,N}_{t'_{min},0}}(Y_j))-(^2\widetilde{X}_i-{^2\widetilde{X}_j})\big||t-t'_{min}|  \notag \\
\le & 2\sup_{k\in \{i,j\}}\sup_{0\le s \le T}|\varphi^N_{s,0}(Y_k)-[{\Psi}^N_{s,0}(X)]_k| |t-t'_{min}| \notag \\
&+  C|^2\widetilde{X}_i-{^2\widetilde{X}_j}||t-t'_{min}|^2 \notag \\
& +2\sup_{k\in \{i,j\}}\sup_{0\le s \le T}|\varphi^N_{s,0}(Y_k)-[{\Psi}^N_{s,0}(X)]_k| |t-t'_{min}| \notag \\
\le & 4 N^{-\delta}\underbrace{\Delta v}_{\le C|^2\widetilde{X}_i-{^2\widetilde{X}_j}|}|t-t'_{min}|+ C|^2\widetilde{X}_i-{^2\widetilde{X}_j}||t-t'_{min}|^2 
\end{align*}
where in the last step we applied \eqref{Gleichung Cor. 3}.\\ 
Thus, in this case the `real' particles fly apart almost like freely moving particles for a possibly short (but $N$- and $X$-independent) time span after a collision exactly like the mean-field particles. Just like in the proof of part (i2) this implies that
\begin{align*}
&\int^{t'_{min}+\frac{1}{C}}_{t'_{min}}|l_N([\Psi^{1,N}_{s,0}(X)]_i-[\Psi^{1,N}_{s,0}(X)]_j)|ds\\
\le & \frac{C}{|^2\widetilde{X}_i-{^2\widetilde{X}_j}|}\min\big(1,\ln(|^1\widetilde{X}_i-{^1\widetilde{X}_j}|\big)\\
\le & \frac{C}{\Delta v}\min\left(1,\ln\left(|^1\widetilde{X}_i-{^1\widetilde{X}_j}|\right)\right)
\end{align*}
and analogously
\begin{align*}
\int^{t'_{min}+\frac{1}{C}}_{t'_{min}}|h_N([\Psi^{1,N}_{s,0}(X)]_i-[\Psi^{1,N}_{s,0}(X)]_j)|ds
\le  \frac{C}{\Delta v}\min\big(\frac{1}{c_N}, \frac{1}{|^1\widetilde{X}_i-{^1\widetilde{X}_j}|}\big).
\end{align*}
where the last step follows again due to \eqref{Gleichung Cor. 3}.\\
Since $t'\in [0,T]$ was chosen arbitrarily and the length of the time interval $\frac{1}{C}$ can be selected independent of $X$ and $N$, this estimate can again be `extended' successively to the whole time period $[0,T]$ such that 
\begin{align*}
\int^{T}_{0}|l_N([\Psi^{1,N}_{s,0}(X)]_i-[\Psi^{1,N}_{s,0}(X)]_j)|ds
\le C\min\left(\frac{ \ln\left(|\Psi^{1,N}_{s,0}(X)]_i-[\Psi^{1,N}_{s,0}(X)]_j|\right)}{\Delta v},\frac{1}{\Delta v}\right)
\end{align*}
and analogously
\begin{align*}
\int^{T}_{0}|h_N([\Psi^{1,N}_{s,0}(X)]_i-[\Psi^{1,N}_{s,0}(X)]_j)|ds
\le  C\min\left(\frac{1}{c_N\Delta v},\frac{1}{\Delta v|\Psi^{1,N}_{s,0}(X)]_i-[\Psi^{1,N}_{s,0}(X)]_j|}\right).
\end{align*}
\end{itemize}
\end{proof}
\section{Derivation of Vlasov equation as the mean-field limit of particle systems with regularised interaction}
For simplification we consider two different groups of particles depending on their distance to other particles. The first set $\mathcal{M}_{b}$ of the bad once includes the particles $j\in \lbrace 1\hdots N\rbrace$ for which there exists a time $t\geq 0$ such that $|\bar{q}_j-\bar{q}_k|\leq N^{-b_{r}}$ and $|\bar{p}_j-\bar{p}_k|\leq N^{-b_{v}}$. They are expected to come very close to other particles with small relative velocity. 
In the end of this section we will show, that typically no bad particle occurs.
The second group consist of the reaming unproblematic good ones, which never come close to each while having small relative velocity are contained in $\mathcal{M}_g=M_b^c$.
Furthermore it depends only on their corresponding mean-field particle whether a particle is considered bad or good.
In the course of a simple notation we introduce collision classes, which turn out to be very important throughout the proof, as each collision class has a different impact on the force term. They are intended to cover all possible ways in which particles can meet and thus also the particle groups can be defined using this notation. 
\begin{definition}
For $r,R,v,V\in \R_{0}^{+}\cup \lbrace\infty\rbrace, t_1,t_2\in \lbrack 0,T\rbrack$ and $Y\in \R^4$ the set $M_{(r,R),(v,V)}^{N,(t_1,t_2)}(Y)\subset \R^4$ is defined as follows:
\begin{align*}
&Z\in M_{(r,R),(v,V)}^{N,(t_1,t_2)}(Y)\Leftrightarrow Z\neq Y \wedge\exists t\in \lbrack t_1,t_2\rbrack:\\
& r\leq \min_{t_1\leq s\leq t_2} |\varphi_{s,0}^1(Z)-\varphi_{s,0}^1(Y)|=|\varphi_{t,0}^1(Z)-\varphi_{t,0}^1(Y)|\leq R\wedge v\leq|\varphi_{t,0}^2(X)-\varphi_{t,0}^2(Y)|\leq V.
\end{align*}
\end{definition}
Here $(\varphi_{s,r}^N)_{s,r\in\R}$ is the one particle mean-field flow related to the considered initial density $k_0$.
For convenience, we will use the following short notation $
M_{R;V}^{N,(t_1,t_2)}(Y):=M_{(0,R),(0,V)}^{N,(t_1,t_2)}(Y),$ 
$M_{(r,R),(v,V)}^{N}(Y):= M_{(r,R),(v,V)}^{N,(0,T)}(Y)$ and
$M_{R,V}^{N}(Y):= M_{(0,R),(0,V)}^{N,(0,T)}(Y).
$

The set $G^N(Y)\subset\R^4$ of mean-field particles with no problematic distance, referred to as mean-field partners of good particles, is defined by
\begin{align}\label{Def:good}
G^N(Y):=(M^{N}_{4r_b,v_b})^c=\left(M^{N}_{4N^{-\frac{3}{5}-\sigma},N^{-\frac{3}{5}-\sigma}}\right)^c.
\end{align}
With this definition, we can partition the set of particles into two groups:
a bad group, where hard collisions are expected, and a group of the remaining good particles. These are defined by
\begin{align*}
\mathcal{M}_{g}^{N}(X) := \left\lbrace i \in {1, \ldots, N} : \forall j \in {1, \ldots, N} \setminus {i}, ; X_j \in G^N(X_i) \right\rbrace, \
\mathcal{M}_{b}^{N}(X):=\{1, \ldots, N\} \setminus \mathcal{M}_{g}^{N}(X).
\end{align*}
By definition, the classification of a particle as good or bad depends exclusively on the state of its associated mean-field particle.
The stopping time for the whole system is given by 
\begin{align}\label{stoppingtime}
\tau^N(X):=  \min(\tau_g^N(X),\tau_b^N(X)) 
\end{align}
where 
\begin{align*}
\tau_g^N&:=  \sup\lbrace t\in \lbrack 0,T\rbrack: \max_{i\in \mathcal{M}_{g}^{N}}\sup_{0\leq s\leq t}|\lbrack \Psi_{s,0}^N(X)\rbrack_i-\varphi_{s,0}^N(X_i)|\leq \delta_g^N=N^{-\frac{2}{5}+2\sigma}\rbrace\\
\tau_b^N&:=\sup\lbrace t\in \lbrack 0,T\rbrack: \max_{i\in \mathcal{M}_{b}^{N}}\sup_{0\leq s\leq t}|\lbrack \Psi_{s,0}^N(X)\rbrack_i-\varphi_{s,0}^N(X_i)|\leq \delta_b^N=N^{-\frac{2}{5}+2\sigma}\rbrace.
\end{align*}
Note that the distinction between good and bad particles is purely technical; for this reason, we have $\delta_g^N=\delta_b^N$.
We will see that configurations fulfilling $\tau^N(X)<T$ become sufficient small in probability for large values of $N$, so that the stopping time typically is not triggered, and hence Theorem \ref{maintheorem} follows.

The main part of the proof is based on application of Gronwall's Lemma to show that $\sup_{0\leq s\leq t}|\lbrack \Psi_{s,0}^N(X)\rbrack_i-\varphi_{s,0}^N(X_i)|_{\infty}$ stays typically small for large $N$.
Therefore we estimate the right derivative of $\sup_{0\leq s\leq t}|\lbrack \Psi_{s,0}^N(X)\rbrack_i-\varphi_{s,0}^N(X_i)|$, which is given by
\begin{align*}
& \frac{d}{dt_+}\sup_{0\le s\le t}|[\Psi^{1,N}_{s,0}(X)]_i-{\varphi^{1,N}_{s,0}}(X_i)|\\ \le &  |[\Psi^{2,N}_{t,0}(X)]_i-{\varphi^{2,N}_{t,0}}(X_i)|\\
\le &|\int_{0}^t\frac{1}{N}\sum_{j\neq i}f^{N}([\Psi^{1,N}_{s,0}(X)]_i-[\Psi^{1,N}_{s,0}(X)]_j)-f^{N}*\widetilde{k}^N_s({\varphi^{1,N}_{s,0}}(X_i))ds|.
\end{align*}
\subsection{Controlling the deviations of good particles}
In this chapter we focus on the case, that the labelled particle $X_i$ is good and use a similar proof technique as presented in \cite{grass,Dustin}.
First we break down the equation in terms of interaction partners. These can themselves be bad or good relative to $X_i$. Of course the set of particles having a bad interaction is empty in this case as having a unpleasant collision is symmetrical and consequently the underlying term will vanish later, but still it will be technically useful to split the equation in that way.
Let $i\in \mathcal{M}_g^N(X)$ and $0\leq t_1\leq t\leq T$ 
 \begin{align}
 & \left|\int_{t_1}^t\frac{1}{N}\sum_{j\neq i}f^{N}\left([\Psi^{1,N}_{s,0}(X)]_i-[\Psi^{1,N}_{s,0}(X)]_j)-f^{N}*\widetilde{k}^N_s({\varphi^{1,N}_{s,0}}(X_i)\right)ds\right|\notag\\
\le & \left|\int_{t_1}^t\frac{1}{N}\sum_{j\neq i}f^{N}([\Psi^{1,N}_{s,0}(X)]_i-[\Psi^{1,N}_{s,0}(X)]_j)\mathds 1_{(G^N(X_i))^c}(X_j)ds\right| \notag\\
& + \left|\int_{t_1}^t\left(\frac{1}{N}\sum_{j\neq i}f^{N}\left([\Psi^{1,N}_{s,0}(X)]_i-[\Psi^{1,N}_{s,0}(X)]_j\right)\mathds 1_{G^N(X_i)}(X_j)  -f^{N}*\widetilde{k}^N_s\left({\varphi^{1,N}_{s,0}}(X_i)\right)\right)ds\right| .\label{termmain}
\end{align}

Using triangle inequality several times on gets that the previous term \eqref{termmain} is bounded by
\begin{align}
 & \big|\int_{t_1}^t\frac{1}{N}\sum_{j\neq i}f^{N}([\Psi^{1,N}_{s,0}(X)]_i-[\Psi^{1,N}_{s,0}(X)]_j)\mathds 1_{(G^N(X_i))^c}(X_j)ds\big| \label{eq:good1} \\
& + \big|\int_{t_1}^t\frac{1}{N}\sum_{j\neq i}\Big(  f^{N}([\Psi^{1,N}_{s,0}(X)]_i-[\Psi^{1,N}_{s,0}(X)]_j)\mathds 1_{G^N(X_i)}(X_j)  \notag \\
& \ \ \ \ -f^{N}(\varphi^{1,N}_{s,0}(X_i)-{\varphi^{1,N}_{s,0}}(X_j))\mathds 1_{G^N(X_i)}(X_j)\Big) ds\big| \label{eq:good2} \\
& +\big| \int_{t_1}^t \frac{1}{N}\sum_{j\neq i}f^{N}(\varphi^{1,N}_{s,0}(X_i)-{\varphi^{1,N}_{s,0}}(X_j))\mathds 1_{G^N(X_i)}(X_j)ds \notag \\
& \ \ \ \ -\int_{t_1}^t\int_{\R^4}f^N({\varphi^{1,N}_{s,0}}(X_i)-{\varphi^{1,N}_{s,0}}(Y))\mathds 1_{G^N(X_i)}(Y)k_0(Y)d^4Yds\big| \label{eq:good3}  \\
& +\big|\int_{t_1}^t\int_{\R^4}f^N({\varphi^{1,N}_{s,0}}(X_i)-{\varphi^{1,N}_{s,0}}(Y))\mathds 1_{G^N(X_i)}(Y)k_0(Y)d^4Yds\notag\\
& \ \ \ \ -\int_{t_1}^tf^{N}*\widetilde{k}^N_s(\varphi^{1,N}_{s,0}(X_i))ds\big| \label{eq:good4} 
\end{align}

\subsubsection{Estimate of Term \ref{eq:good1} and Term \ref{eq:good4}}
Recall that $i\in \mathcal{M}_g^N(X)$ and note that the set $(G^N(X_i))^c=M^N_{4r_b,v_b}$ includes all particles, which come close to $X_i$ while having small relative velocity. Thus the characteristic function $\mathds 1_{(G^N(X_i))^c}(X_j)=0$ for $i\in \mathcal{M}_g^N(X)$ and therefore term (\ref{eq:good1}) vanishes and we are left to estimate term (\ref{eq:good4}). Due to the properties of $\varphi^{1,N}_{s,0}$ the following holds
\begin{align*}
 f^{N}*\widetilde{k}^N_s({\varphi^{1,N}_{s,0}}(X_i))
= \int_{\R^4}f^{N}({\varphi^{1,N}_{s,0}}(X_i)-{^1Y})k^N_s(Y)d^4Y
= \int_{ \R^4}f^{N}({\varphi^{1,N}_{s,0}}(X_i)-{\varphi^{1,N}_{s,0}}(Y))k_0(Y)d^4Y.
\end{align*}
So we get for term (\ref{eq:good4})
\begin{align*}
& \big|\int_{t_1}^t\int_{ \R^4}f^{N}(\varphi^{1,N}_{s,0}(X_i)-{\varphi^{1,N}_{s,0}}(Y))k_0(Y)\mathds 1_{G^N(X_i)}(Y)d^4Yds \notag 
-\int_{t_1}^tf^{N}*\widetilde{k}^N_s({\varphi^{1,N}_{s,0}}(X_i))ds\big| \notag \\
= & \big|\int_{t_1}^t\int_{ \R^4}f^{N}(\varphi^{1,N}_{s,0}(X_i)-{\varphi^{1,N}_{s,0}}(Y))k_0(Y)(\mathds 1_{G^N(X_i)}(Y)-1)d^4Yds\big| \notag \\
\le & T\|f^{N}\|_{\infty}\int_{ \R^4}\mathds 1_{(G^N(X_i))^c}(Y)k_0(Y)d^4Y \notag \\
\le & TN^{\beta}\mathbb{P}\big(Y \in \R^4:Y\notin G^N(X_i) \big)\\
\le&TN^{\beta}\mathbb{P}\big(Y\in \R^4:Y\in M^N_{r_b,v_b}(X_i)  \big)\\
\le& TN^{\beta-b_r-3b_v}  \le CN^{-\frac{2}{5}-4\sigma}
\end{align*}
\subsubsection{Law of large numbers to estimate Term \ref{eq:good2} and Term \ref{eq:good3}}
For the remaining terms \ref{eq:good2} and \ref{eq:good3}  we need a version of law of large numbers which takes into account the different types of collision classes which could occur. Each collision type has a different impact on the force and a certain probability. For that reason it is useful for the estimates to distinguish between them.
\begin{theorem}\label{Theorem LolN} Let $\delta,C>0$, $N{\in \N} $ and let $(X_{k})_{k\in \N}$ be a sequence of i.i.d. random variables $X_k:\Omega \to \R^4$ distributed with respect to a probability density $k\in \mathcal{L}^1(\R^4)$. Moreover, let $(M^N_i)_{i\in I}$ be a family of (possibly $N$-dependent) sets $M^N_i\subseteq \R^4$ fulfilling $\bigcup_{i \in I}M^N_i=\R^4$ where $|I|<C$ and $h_N:= \R^4\to \R$ measurable functions which fulfil on the one hand $\|h_N\|_{\infty}\le CN^{1-\delta}$ and on the other hand
\[\max_{i \in I}\int_{M^N_i}h_N(X)^2k(X)d^4X\le CN^{1-\delta}.\]
 Then for any $\gamma>0$ there exists a constant $C_{\gamma}>0$ such that for all $N\in\N$
	\begin{align} \mathbb{P}_t\Bigl[ \Bigl\lvert \frac{1}{N}\, \sum\limits_{j\neq i}^N h_N(X_k) -\int_{\R^4} h_N(X)k_t (X)d^4X \Bigr\rvert \geq 1\Bigr] \leq \frac{C_{\gamma}}{N^\gamma}.
	\end{align}
\end{theorem}	
\begin{proof}
	\noindent By Markov's inequality, we have for every $M \in \N$: 
	\begin{align} \mathbb{P}_t\Bigl[ \Bigl\lvert \frac{1}{N}\, \sum\limits_{j\neq i}^N h_N(X_k) -\int_{\R^4} h_N(X)k_t (X)d^4X \Bigr]\rvert\leq& \mathbb{E}\Bigl[ N^{-2M}\, \Bigl\lvert \frac{1}{N}\, \sum\limits_{j=1}^N h_N(X_k) - \int_{\R^4} h_N(X)k_t (X)d^4X \Bigr\rvert^{2M} \Bigr],
	\end{align}
	where $\mathbb{E}\lbrack\cdot\rbrack$ denotes the expectation with respect to the N-fold product of $k$.
	
	\noindent Let $\mathcal{M} :=  \lbrace \mathbf{\gamma} \in \N_0^N \mid \lvert \mathbf{\gamma} \rvert = 2M \rbrace$ the set of multiindices $\mathbf{\gamma} = (\gamma_1, ..., \gamma_N)$ with $\sum\limits_{j=1}^{N} \gamma_j = 2M$. Let
	\begin{equation} G_\mathbf{\gamma}(X) :=  \prod \limits_{j=1}^N \bigl( h_N(X_j) -\int_{\R^4} h_N(X)k_t (X)d^4X)^{\gamma_i}. \end{equation}
	\noindent Then 
	\begin{equation} N^{-2M}\mathbb{E} \Bigr[ \Bigl(\sum\limits_{k=1}^N h_N(X_k) -\int_{\R^4} h_N(X)k_t (X)d^4X\bigr) \Bigr)^{2M}\Bigr] \leq N^{-2M}\sum_{\gamma_1,\hdots,\gamma_N\in\mathcal{M}} \mathbb{E} \Bigr[ \Bigl(G_\mathbf{\gamma}(X))^{\gamma_k}\Bigr]. \end{equation}
	
	\noindent Now we note that  $\mathbb{E}(G_\mathbf{\gamma }) = 0$ whenever there exists a $1 \leq j \leq N$ such that $\gamma_j =1$. This can be seen by integrating the j'th variable first.\\

	\noindent For the remaining  terms, we have for any $1 \leq m \leq M$:
	\begin{align*} \lvert \bigl( h_N(X_j) -\int_{\R^4} h_N(X)k_t (X)d^4X\bigl)^{\gamma_i}\rvert \leq2^{\gamma_k} |h_N(X_j)|^{\gamma_k} +|\int_{\R^4} h_N(X)k_t (X)d^4X| ^{\gamma_k} .
	\end{align*}
	Now as $||h_N||\leq CN^{1-\delta}$ it follows for $m\geq 2$ 
	
	\begin{align*}
	&\int_{\R^4} \lvert h\rvert^m (X)k_t(X) \, \mathrm{d}^4X \leq C \max_{i\in I}\int_{M_i^N} \lvert h\rvert^m (X)k_t(X) \, \mathrm{d}^4X\\
	\leq &C ||h_N||_{\infty}^{m-2}\max_{i\in I}\int_{M_i^N} h_N(X)^2k(X)\mathrm{d}^4X\leq C (C^{m-2}N^{(m-2)(1-\delta)})(CN^{1-\delta})
	\end{align*} 
	Now let $R:= \sqrt{\int_{\R^4}h_N^2(X)\mathrm{d}^4X}$ then it holds that
	\begin{align*}
	&\int_{\mathbb{R}^4} \lvert h_N(X)\rvert k_t(X) \, 1\mathrm{d}^4X \leq \frac{1}{R}\underbrace{ \int_{\mathbb{R}^4}  h_N^2 (X)k_t(X) }_{=R^2}+\underbrace{ \int_{\mathbb{R}^4}  |h_N (X)|\mathds 1_{\lbrack 0,R\rbrack}(h_N(X))k_t(X) }_{\leq R}
\\
&\leq 2\big(C\max_{i\in I}\int_{M_i^N} h_N^2(X)k_t(X)\mathrm{d}^4X\big)^{\frac{1}{2}}	\leq C M^{\frac{1}{2}(1-\delta)}.
	\end{align*}

 Since the constraints on the maps $h_N$ are more restrictive the larger the value of $\delta$ is chosen, we can limit the considered values to (for example) $(0,1]$. If we identify additionally $|\gamma|:= |\{i\in\{1,...,N\}: \gamma_i\neq 0\}|$ and recall that only tuples matter where $\gamma_i \neq 1\ $ for all $i\in \{1,...,N\}$ as well as $\sum_{i=1}^N\gamma_i=2M$, then application of these estimates and relation above yields that. For the other multiindices, we get (using that the particles are statistically independent):
	\begin{align*}\label{LLNestimate} \mathbb{E}_t (G^\mathbf{\gamma}) \leq  &\prod \limits_{j=1:\gamma_i\geq 2}^N  \big(C^{\gamma_i}N^{(\gamma_i-2)(1-\delta)}N^{1-\delta} \big)\leq C^{2M}N^{2M(1\delta)}N^{|\gamma|(\delta-1)}
	\end{align*}
 Finally, we observe that for any $l \geq 1$, the number of multiindices $\mathbf{\gamma} \in \mathcal{M} $ with $\#\mathbf{\gamma} = l$ is bounded by
	\begin{equation*} \sum\limits_{\#\mathbf{\gamma} = l} 1 \leq \binom{N}{l} (2M)^l  \leq (2M)^{2M} N^l. \end{equation*}
	 Thus
	\begin{align*}  \frac{1}{N^{2M}} \sum\limits_{\mathbf{\gamma} \in \mathcal{M}} \mathbb{E}(G^\mathbf{\gamma})
	\leq  \frac{N^{2M(1-\delta)}}{N^{2M}}\, \sum_{\gamma\in\mathcal{M}} C^{M}N^{|\gamma|(\delta-1)}
	\leq (CM)^{M} N^{-\delta M},
	\end{align*}
	where $C$ is some constant depending on $M$ and 
	and choosing $M$ arbitrary large proofs the Theorem.
\end{proof}
\subsubsection{Estimate of Term \ref{eq:good3}}
To show that the third term (\ref{eq:good3}), respectively 
\begin{align*}
&\left| \int_{t_1}^t\left( \frac{1}{N}\sum_{j\neq i}f^{N}(\varphi^{1,N}_{s,0}(X_i)-{\varphi^{1,N}_{s,0}}(X_j))\mathds 1_{G^N(X_i)}(X_j)\right.\right.\\
 & \ \ \ \ \ \ \ \ \ -\left.\left.\int_{\mathbb{R}^4}f^N({\varphi^{1,N}_{s,0}}(X_i)-{\varphi^{1,N}_{s,0}}(Y))\mathds 1_{G^N(X_i)}(Y)k_0(Y)d^4Y\right)ds\right|
\end{align*}
  stays small for typical initial data. Therefore we define for arbitrary $Y\in\mathbb{R}^4$ the function (based on the function from the Theorem \ref{Theorem LolN})
\begin{align}
h_{1,N}^t(y,\cdot):\mathbb{R}^4\rightarrow\mathbb{R}^2,X\mapsto N^{\alpha}\int_{0}^{t}f^{N}(\varphi^{1,N}_{s,0}(Y)-{\varphi^{1,N}_{s,0}}(X))ds \mathds 1_{G^N(Y)}(X),
\end{align}	\label{h_{1,N}}
with $\alpha=\frac{2}{5}-2\sigma$.
The function $h_{1,N}^t(Y,\cdot)$ does not map to $\mathbb{R}$ but it can be applied on each component separately and thus fullfills the assumptions of Theorem \ref{Theorem LolN}.

We are left to check if the assumptions of Theorem \ref{Theorem LolN} on the force term are fulfilled. We obtain by Corollary \ref{corollary phi and  psi} and Lemma \ref{Prob of group} for $0\leq v\leq V$, $0\leq r\leq R$ the following:
\begin{align}\label{Var(f)}
&\int_{M^N_{(r,R),(v,V)}(Y)}\left(\int_{0}^t |f^{N}(\varphi^{1,N}_{s,0}(Z)-{\varphi^{1,N}_{s,0}}(Y))|ds\right)^2k_0(Z) d^4Z \notag \\
\le & C\left(\min\left(\frac{1}{\Delta r},\frac{1}{c_N},\frac{\ln(\frac{\Delta v}{\Delta r})}{\Delta v},\frac{1}{\Delta v}\right)\right)^2\int_{M^N_{(r,R),(v,V)}(Y)}k_0(Z) d^4Z\notag \\
\le &C\min\left(\frac{1}{\Delta r^2},\frac{1}{\Delta c_N^2},\frac{\ln(\frac{\Delta v}{\Delta r})^2}{\Delta v^2},\frac{1}{\Delta v^2}\right)\min \left(1,R,RV^3+R^2\max\left(V^2,R^2\right)\right)\notag \\
\le &C \min\left(\frac{1}{\Delta r^2},\frac{1}{ c_N^2},\frac{\ln(\frac{\Delta v}{\Delta r})^2}{\Delta v^2},\frac{1}{\Delta v^2},\frac{R}{\Delta r^2},\frac{R}{ c_N^2},\frac{R\ln(\frac{\Delta v}{\Delta r})^2}{\Delta v^2},\frac{R}{\Delta v^2},\right.\\
&\quad\left.\frac{RV^3+R^2\max(V^2,R^2)}{\Delta r^2},\frac{RV^3+R^2\max(V^2,R^2)}{c_N^2},\right.\notag\\&\quad\left.\frac{RV^3+R^2\max(V^2,R^2)\ln(\frac{\Delta v}{\Delta r})^2}{\Delta v^2},\frac{RV^3+R^2\max(V^2,R^2)}{\Delta v^2} \right)\notag.
\end{align}
Now we will defind a suitable cover of $\mathbb{R}^4$, i.e. the collision classes, in order to apply the Theorem \ref{Theorem LolN}. The classes are chosen finer as the collision strength becomes larger. If the particles keep distance of order 1 even no splitting will be necessary.
Let therefore be $k,l\in\mathbb{Z},N\in\N\setminus \lbrace 1\rbrace,\eta>0$ and $0\leq r,v\leq 1$ and the family of sets are given by

\begin{align}\label{family of sets}
(i)\ &M_{(0,r)(0,v)}^{N}(Y)  &(ii) \  &M_{(0,r)(N^{l\eta}v,N^{(l+1)\eta}v)}^{N}(Y)& \notag \\
(iii) \ &M_{(0,r)(1,\infty)}^{N}(Y)  & (iv) \ &M_{(N^{k\eta}r,N^{N(k+1)\eta}r)(0,v)}^{N}(Y)& \notag\\
(v) \ &M_{(N^{k\eta}r,N^{(k+1)\eta}r)(N^{l\eta}v,N^{(l+1)\eta}v)}^{N}(Y)& (vi) \ &M_{(N^{k\eta}r,N^{(k+1)\eta}r)(1,\infty)}^{N}(Y)\notag \\
(vii) \ &M_{(N^{-\eta},\infty)(0,\infty)}^{N}(Y), &
\end{align}
for $ 0\leq k\leq \lfloor\frac{\ln(\frac{1}{r})}{\eta\ln(N)}\rfloor,0\leq l\leq \lfloor\frac{\ln(\frac{1}{v})}{\eta\ln(N)}\rfloor$.
In this case we choose $r=v=N^{-\beta}$ and the number of sets belonging to this list is some integer $I_{\eta}$ independent of $N$.\\
We will apply \eqref{Var(f)} for each collision class family and for $0\leq k,l\leq \lfloor\frac{\beta}{\eta}\rfloor$ we receive the bounds 
\begin{align*}
(i)\ &\frac{(N^{-\beta})^4}{(N^{-\beta})^2}=N^{-2\beta} 
& (ii)\  &\frac{N^{-\beta}N^{3[(l+1)\eta-\beta]}}{N^{2(l\eta-\beta)}}=N^{-2\beta+l\eta+3\eta} & \\
(iii)\ &\ln(N^{-\beta})^2N^{-2\beta} 
&\!\!\! (iv)\  &N^{2(k+1)\eta}N^{-2\beta}=N^{-2\beta+2k\eta+2\eta}& \\
(v)\ &\frac{N^{(k\eta+\eta-\beta)}N^{3(l\eta+\eta-\beta)}}{N^{2(l\eta -\beta)}}+\frac{N^{4(k\eta-\beta)}}{N^{2(k\eta -\beta)}}  & (vi)\ &N^{2(k+1)\eta}N^{-2\beta}=N^{-2\beta+2k\eta+2\eta}\\&=N^{-2\beta+l\eta+k\eta+4\eta}+N^{-2\beta+2k\eta+4\eta}  \\
(vii)\  &N^{2\eta}.
\end{align*}
All these terms are bounded by $N^{2\eta}$.
For law of large numbers argument we need
$\|h_{1,N}\|_{\infty}\le CN^{1-\delta}$ and 
$\max_{i \in I}\int_{M^N_i}h_{1,N}(X)^2k(X)d^4X\le CN^{1-\delta}$.
Due to the estimates above it follows for all $i\in I$
\begin{align*}
& \int_{M^N_{(r_i,R_i),(v_i,V_i)}(Y)}h^t_{1,N}(Y,Z)^2k_0(Z)d^{4}Z
\le CN^{2\alpha-2\eta},
\end{align*} 
for $\eta>0$ chosen small enough and $\alpha=\frac{2}{5}-2\sigma$. Thus for $N$ large enough  it follows $2\alpha-2\eta<1$ and the first assumption of Theorem \ref{Theorem LolN} is fulfilled. 

It holds due to Corollary \ref{corollary phi and  psi} that for a point in time $t_{min}$, where the mean-field particles are close
\begin{align}
& \int_{0}^t|f^{N}(\varphi^{1,N}_{s,0}(Y)-{\varphi^{1,N}_{s,0}}(Z))|\mathds 1_{G^N(Z)}(Y)ds \notag \\
\le & \min\left(\frac{Ct}{|\varphi^{1,N}_{t_{min},0}(Y)-{\varphi^{1,N}_{t_{min},0}}(Z)|},\frac{C}{| \varphi^{2,N}_{t_{min},0}(Y)-{\varphi^{2,N}_{t_{min},0}}(Z)|},\frac{C\ln\left(|\varphi^{1,N}_{t_{min},0}(Y)-{\varphi^{1,N}_{t_{min},0}}(Z)|\right)}{| \varphi^{2,N}_{t_{min},0}(Y)-{\varphi^{2,N}_{t_{min},0}}(Z)|}\right)\mathds 1_{G^N(Z)}(Y) .
\end{align} 
This is where we break down the time integral into several parts. If $v$ is large, the assumptions of Theorem \ref{Theorem LolN} are fulfilled directly. If $v$ is small we made use of the fact that the collision time is not very large.
The following implication holds due to the definition of $$
G^N(Z):= \left(M^N_{4N^{-\frac{3}{5}-\sigma},N^{-\frac{3}{5}-\sigma}}(Z)\right)^c.
$$ We have
\begin{align*}
 x_{min}&:= |\varphi^{1,N}_{t_{min},0}(Y)-{\varphi^{1,N}_{t_{min},0}}(Z)|\leq N^{-r_b}
\Rightarrow v_{min}:= |\varphi^{2,N}_{t_{min},0}(Y)-{\varphi^{2,N}_{t_{min},0}}(Z)|\geq N^{-v_b}\\
N^{-r_b}\leq x_{min}&\Rightarrow v_{min}\in \mathbb{R}^+.
\end{align*}
and thus the term is bounded in the first case (small inter particle distance) by $
CtN^{v_b}
$ and
for the second case is bounded by 
$
CN^{b_r}.
$           
For sufficiently small $\sigma > 0$ and $\alpha=\frac{2}{5}-2\sigma$ we get
\begin{align*}
||h_{1,N}^t(Y,\cdot)||_{\infty}\leq N^{\alpha}C (N^{r_b}+N^{v_b})\leq CN^{1-\sigma}.
\end{align*}
We now apply our estimate on $h_{1,N}^t(y)$ defined in \eqref{h_{1,N}} to control term \ref{eq:good3}.
 Therefore we introduce the set $\mathcal{B}_{1,i}^{N,\sigma}\subset\mathbb{R}^{4N},i\in\lbrace 1,\hdots, N\rbrace$:
\begin{align}
\begin{split}
& X\in \mathcal{B}_{1,i}^{N,\sigma}\subseteq \mathbb{R}^{4N}\\ \label{set.B_1 good}
\Leftrightarrow & \exists t_1,t_2\in [0,T]:\\
& \Big|\frac{1}{N}\sum_{j\neq i}\int_{t_1}^{t_2}f^{N}(\varphi^{1,N}_{s,0}(X_i)-{\varphi^{1,N}_{s,0}}(X_j))\mathds 1_{G^N(X_i)}(X_j)ds\\
&  -\int_{\mathbb{R}^4}\int_{t_1}^{t_2} f^{N}(\varphi^N_{s,0}(X_i)-\varphi^N_{s,0}(Y))\mathds 1_{G^N(X_i)}(Y)dsk_0(Y)d^4Y\Big| >N^{-\alpha}.
\end{split}
\end{align}
As law of large numbers only makes statements for certain points in time but on the other hand tells us that that at one considered moment large fluctuations are extremely unlikely and furthermore on very short time intervals fluctuations cannot change significantly since the force is bounded due to the cut off by $N^{\beta}$, so this problem can be solved. 
By the definition of the set $\mathcal{B}_{1,i}^{N,\sigma}$ and by the fact that for any continuous map  $a:\mathbb{R}\rightarrow\mathbb{R}^m,m\in\N,t_1,t_2\in\lbrack 0,T\rbrack$ it holds that
\begin{align*}
& \big|\int_{t_1}^{t_2 }a(s)ds\big| = \big|\int_{0}^{t_2 }a(s)ds-\int_{0}^{t_1 }a(s)ds\big| \notag \\
\le & \big|\int_{0}^{\lfloor \frac{t_2}{\delta_N}\rfloor \delta_N }a(s)ds\big|+\int_{\lfloor \frac{t_2}{\delta_N}\rfloor \delta_N }^{t_2}|a(s)|ds +  \big|\int_{0 }^{\lfloor \frac{t_1}{\delta_N}\rfloor \delta_N}a(s)ds\big|+\int_{\lfloor \frac{t_1}{\delta_N}\rfloor \delta_N }^{t_1}|a(s)|ds \notag \\
\le & 2\max_{k\in \{0,...,\lfloor \frac{T}{\delta_N}\rfloor \}}\Big(\big|\int_{0 }^{k \delta_N}a(s)ds\big| +\int_{k \delta_N }^{(k+1) \delta_N }|a(s)|ds\Big),
\end{align*}
it follows for $\delta_N>0$ that:
\begin{align*}
& X\in \mathcal{B}_{1,i}^{N,\sigma} \\ 
\Rightarrow & \exists k\in \{0,...,\lfloor \frac{T}{\delta_N } \rfloor \}:  \notag \\
&  \Big(\big|\int_{0}^{k \delta_N }\Big(\frac{1}{N}\sum_{j\neq i}f^{N}(\varphi^{1,N}_{s,0}(X_i)-{\varphi^{1,N}_{s,0}}(X_j))\mathds 1_{G^N(X_i)}(X_j)   \\
& -\int_{\mathbb{R}^4} f^{N}(\varphi^{1,N}_{s,0}(X_i)-{\varphi^{1,N}_{s,0}}(Y))\mathds 1_{G^N(X_i)}(Y)k_0(Y)
d^4Y\Big) ds\big| \geq \frac{N^{-\alpha}}{4}\Big) \ \vee  \\
& \Big(\int_{k \delta_N }^{(k+1) \delta_N }\Big(\big|\frac{1}{N}\sum_{j\neq i}f^{N}(\varphi^{1,N}_{s,0}(X_i)-{\varphi^{1,N}_{s,0}}(X_j))\mathds 1_{G^N(X_i)}(Y)
\big|  \\
& +\big|\int_{\mathbb{R}^4} f^{N}(\varphi^{1,N}_{s,0}(X_i)-{\varphi^{1,N}_{s,0}}(Y))\mathds 1_{G^N(X_i)}(Y)k_0(Y)
d^4Y\big|\Big) ds \geq \frac{N^{-\alpha}}{4}\Big)
\end{align*}

If we choose $\delta_N:= \frac{N^{-\alpha}}{8||f_N||_{\infty}}\leq CN^{-\alpha-\beta}$ the second constraint of the asumption is true.
As $\beta=2$ and $\alpha=\frac{2}{5}-2\sigma$ the number of such events is bounded by $\lfloor\frac{T}{\delta_N}\rfloor +1\le CN^{\frac{12}{5}-2\sigma}  $ 
 and thus it holds for all $N\in\mathbb{N}$ that
\begin{align*}
\mathbb{P}\left(\exists i\in\lbrace 1,\hdots,N\rbrace:X\in\mathcal{B}_{1,i}^{N\sigma}\right)&\leq N \mathbb{P}\left(X\in\mathcal{B}_{1,i}^{N\sigma}\right)\\
&\leq N \left(CN^{\frac{12}{5}-2\sigma}\left(C_{\gamma+\frac{17}{5}-2\sigma}N^{-(\gamma+\frac{17}{5}-2\sigma)}\right)\right)\\
&\leq \tilde{C}_{\gamma}N^{-\gamma}.
\end{align*}
For typical initial data and large enough $N\in\N$ term (\ref{eq:good3}) stays smaller than $N^{-\alpha}=N^{-\frac{2}{5}-2\sigma}.$

\subsubsection{Estimate of Term \ref{eq:good2}}
To estimate term (\ref{eq:good2}), i.e. the difference of the real force acting on the real particles and the real force acting on the mean-field particles
\begin{align*}
 \left|\int_{t_1}^t\frac{1}{N}\sum_{j\neq i}\left(  f^{N}\left([\Psi^{1,N}_{s,0}(X)]_i-[\Psi^{1,N}_{s,0}(X)]_j\right)
-f^{N}\left(\varphi^{1,N}_{s,0}(X_i)-{\varphi^{1,N}_{s,0}}(X_j)\right)\right)\mathds 1_{G^N(X_i)}(X_j) ds\right|. 
\end{align*} 
We abbreviate the following notation for the allowed difference between mean-field particle and the real one, depending on the group membership. We allow less control if the particle is bad but have strict requirements if the particle is good.
 $ \Delta^N_g(t,X)$ describes the largest spatial deviation of the good particles, $\Delta^N_b(t,X)$ the corresponding value for the bad ones. 
\begin{align*}
 \Delta^N_g(t,X)&:= \max_{j\in \mathcal{M}^N_g(X)}\sup_{0\le s\le t}|[\Psi^{1,N}_{s,0}(X)]_j-{\varphi^{1,N}_{s,0}}(X_j)| \notag \\
 \Delta^N_b(t,X)&:= \max_{j\in \mathcal{M}^N_b(X)}\sup_{0\le s\le t}|[\Psi^{1,N}_{s,0}(X)]_j-{\varphi^{1,N}_{s,0}}(X_j)|.  \notag 
\end{align*}
We further introduce a subgroup of the good particles 
$\widetilde{G}^N(\cdot):= G^N(\cdot)\cap \left(M^N_{2N^{-\frac{2}{5}+2\sigma},\infty}(\cdot)\right)^c$ which will be used to simplify the forthcoming estimates.
By definition of $\widetilde{G}^N(\cdot)$ (applied for the first inequality) and the stopping time $\tau^N(X)$
\begin{align*}
\tau_g^N:= \sup\lbrace t\in \lbrack 0,T\rbrack: \max_{i\in \mathcal{M}_{g}^{N}}\sup_{0\leq s\leq t}|\lbrack \Psi_{s,0}^N(X)\rbrack_i-\varphi_{s,0}^N(X_i)|\leq \delta_g^N\\
\tau_b^N:= \sup\lbrace t\in \lbrack 0,T\rbrack: \max_{i\in \mathcal{M}_{b}^{N}}\sup_{0\leq s\leq t}|\lbrack \Psi_{s,0}^N(X)\rbrack_i-\varphi_{s,0}^N(X_i)|\leq \delta_b^N
\end{align*}
as well as $\tau^N(X):= \min(\tau_g^N(X),\tau_b^N(X))$ with $\delta_g^N=\delta_b^N=N^{-\frac{2}{5}+2\sigma}$ 
 it holds for $X_j\in \widetilde{G}^N(X_i)$ and times $s\in [0,\tau^N(X)]$ that $
 \max\big( 2N^{-\beta},\frac{2}3|{\varphi^{1,N}_{s,0}}(X_j)-{\varphi^{1,N}_{s,0}}(X_i)|\big)\geq  \max \big( 2N^{-\beta},2N^{-\frac{2}{5}+2\sigma}\big) \geq 2\Delta^N_g(t,X).$
 
In the next step we subdivide the sum according to whether the particle interacting with $i$ is itself  bad or good.
Furthermore, the map $g^N$ was defined such that $|f^N(q+\delta)-f^N(q)|\le g^N(q)|\delta|$ for $q,\delta \in \mathbb{R}^2$ where $ \max\big(2 N^{-\beta},\frac{2}{3}|q|\big)\geq |\delta|$ (see definition \eqref{force g}) and thus the subsequent estimates are fulfilled for all $0\le t_1\le t\le \tau^N(X)$.
\begin{align}
\eqref{eq:good3}\leq & \int_{0}^t\Big(\frac{1}{N}\sum_{\substack{j\neq i\\ j\in \mathcal{M}^N_{b}(X)}}\Big(\big|f^{N}([\Psi^{1,N}_{s,0}(X)]_j-[\Psi^{1,N}_{s,0}(X)]_i) \notag  \\
&- f^{N}(\varphi^{1,N}_{s,0}(X_j)-{\varphi^{1,N}_{s,0}}(X_i))\big|\Big)\mathds 1_{G^N(X_i)}(X_j)\Big) ds\label{term b}\\
&+\int_{0}^t\Big(\frac{1}{N}\sum_{\substack{j\neq i\\ j\in \mathcal{M}^N_g(X)}}\Big(\big|f^{N}([\Psi^{1,N}_{s,0}(X)]_j-[\Psi^{1,N}_{s,0}(X)]_i)\big| \notag  \\
&+ \big|f^{N}(\varphi^{1,N}_{s,0}(X_j)-{\varphi^{1,N}_{s,0}}(X_i))\big|\Big)\mathds 1_{G^N(X_i)\cap M^N_{2N^{-\frac{2}{5}-2\sigma},\infty}(X_i)}(X_j)\Big) ds \label{term c}\\
 & + \int_{0}^t\frac{2}{N}\sum_{\substack{j\neq i\\ j\in \mathcal{M}^N_g(X)}}g^{N}(\varphi^{1,N}_{s,0}(X_j)-{\varphi^{1,N}_{s,0}}(X_i))\Delta^N_g(s,X)\mathds 1_{\widetilde{G}^N(X_i)}(X_j) ds. \label{term c2}
\end{align}
To estimate Term \ref{term c2} we define a set of atypical events
\begin{align}
\begin{split}
& X\in \mathcal{B}_{2,i}^{N,\sigma}\subseteq \mathbb{R}^{4N} \label{def.B_2 good}\\
\Leftrightarrow & \exists t_1,t_2\in [0,T]:\\
& \Big|\frac{1}{N}\sum_{j\neq i}\int_{t_1}^{t_2}g^{N}({\varphi^{1,N}_{s,0}}(X_j)-{\varphi^{1,N}_{s,0}}(X_i))\mathds 1_{\widetilde{G}^N(X_i)}(X_j)ds\\
&  -\int_{\mathbb{R}^4}\int_{t_1}^{t_2} g^{N}(\varphi^{1,N}_{s,0}(Y)-{\varphi^{1,N}_{s,0}}(X_i))\mathds 1_{\widetilde{G}^N(X_i)}(Y)dsk_0(Y)
d^4Y\Big| > 1
\end{split} 
\end{align}
For $Y,Z\in \mathbb{R}^4$ it holds by definition of $\widetilde{G}^N(\cdot)$ and the definition of $g^N$ (defined in Definition \ref{force g}) that
\begin{align}
& \int_{0}^t g^{N}(\varphi^{1,N}_{s,0}(Y)-{\varphi^{1,N}_{s,0}}(Z)) \mathds 1_{\widetilde{G}^N(Z)}(Y)ds \notag \\
\le  & C
N^{\frac{2}{5}-2\sigma}\int_{0}^t|f^{N}(\varphi^{1,N}_{s,0}(Y)-{\varphi^{1,N}_{s,0}}(Z))|\mathds 1_{G^N(Z)}(Y)ds. 
\end{align}
This term fulfils the assumptions of the Theorem \ref{Theorem LolN} like in the previous section and the same reasoning as used previously for the map $h^t_N(Y,\cdot)$ works to show that for an arbitrary $\gamma>0$ there exists $C_{\gamma}>0$ such that for all $N\in \mathbb{N}$
\begin{align}
\mathbb{P}\big(\exists i\in \{1,...,N\}:X\in \mathcal{B}_{2,i}^{N,\sigma}\big)\le C_{\gamma} N^{-\gamma}.
\end{align}
The next step is to derive an upper bound for Term \eqref{term b} and Term \eqref{term c}. Term \eqref{term b} models the interaction of a good particle with a bad particle. We will show that, in general, bad particles do not infect the good ones.
 
After introducing the function $h^t_N(Y,\cdot)$ according to Theorem~\ref{Theorem LolN}, as in the previous part, defined by
\begin{align}
h_N^t(Y,\cdot) : \mathbb{R}^4 \rightarrow \mathbb{R}^2, \quad Z \mapsto N^{\alpha} \int_{0}^{t} f^{N}\big(\varphi^{1,N}_{s,0}(Y) - \varphi^{1,N}_{s,0}(Z)\big) \, ds \, \mathds{1}_{G^N(Y)}(Z),
\end{align}
with $\alpha = \frac{2}{5} - 2\sigma$, we introduced a family of collision classes $\big(M^N_{(r_i,R_i),(v_i,V_i)}(Y)\big)_{i \in I_\delta}$ which forms a cover of $\mathbb{R}^4$. Subsequently, we verified that the function $h^t_N(Y,\cdot)$, together with this covering, satisfies the assumptions of Theorem~\ref{Theorem LolN}. This allows us to derive an upper bound for the term~\eqref{term b}.
Therefore let $\big(M^N_{(r_i,R_i),(v_i,V_i)}(Y)\big)_{i \in I_{\delta}}$ denote the family of sets defined in \eqref{family of sets} but this time for the parameters $r:= r_b \ v:= 4v_b$. Thus we define for $i\in  \{1,...,N\}$ the set $ \mathcal{B}_{3_b,i}^{N,\sigma}\subseteq \mathbb{R}^{4N}$ as follows:
\begin{align}
\begin{split}
& X\in \mathcal{B}_{3_b,i}^{N,\sigma} \subseteq \mathbb{R}^{4N} \\
\Leftrightarrow & \exists l\in I_{\sigma}: \Big(R_l\neq \infty\ \land\\ &\sum_{j \in \mathcal{M}^N_b(X)}\mathbf{1}_{M^N_{(r_l,R_l),(v_l,V_l)}(X_i)}(X_j) \geq  N^{\sigma\frac{3}{4}}\big\lceil N^{\frac{3}{4}} R_l^2\min\big(\max(V_l,R_l),1\big)^4\big\rceil\Big) \ \vee \label{def.B_3b}\\
& \sum_{j \in \mathcal{M}^N_b(X)}1=| \mathcal{M}^N_b(X)|\geq N^{2}v_{b}^3r_{b}=N^{-\frac{2}{5}-4\sigma}.
\end{split}
\end{align} 
The last line gives an estimate of the absolute number of bad particles and the line above an estimate of how many bad particles come close to a good one given a certain inter-particle distance and velocity and we can see at this point that typically no bad particle occurs.
\subsubsection{Estimate of Term \eqref{term b}}
We now derive an upper bound for term \eqref{term b} under the condition that $X\in \big(\mathcal{B}_{3_{b},i}^{N,\sigma}\big)^c$ and prove later that $\mathbb{P}\big( X\in \mathcal{B}_{3_{b},i}^{N,\sigma}\big)$ gets small as $N$ increases.
To this end, we abbreviate for $0\le r\le R $ and $0\le v \le V$
$$ \widetilde{M}^N_{(r,R),(v,V)}(X_i):= G^N(X_i)\cap M^N_{(r,R),(v,V)}(X_i)$$
to distinguish between the collision classes.
In the analysis of Term~\eqref{term b}, we restrict the parameters $r$ and $R$ to satisfy the condition
\begin{align}
\big(r = 0 \land R = 4\delta^N_{b} = 4N^{-\frac{2}{5} -2\sigma}\big) \;\vee\; \big( r \geq 4\delta^N_{b} \land R = N^{\sigma} r \big), \label{cond M delta bad}
\end{align}
as, within the family of collision sets $\big(M^N_{(r_i,R_i),(v_i,V_i)}(Y)\big)_{i \in I_{\sigma}}$, only indices corresponding to such parameter choices contribute significantly to the estimate. Consequently, it suffices to consider exclusively those $(r,R)$ combinations that comply with the relation~\eqref{cond M delta bad}.

Before stopping time is triggered it holds that $$\sup_{0\le s \le t}|\Psi^N_{s,0}(X)-\Phi^N_{s,0}(X)|_{\infty}\le \delta^N_{b}=N^{-\frac{2}{5}-2\sigma}$$ and thus we obtain that for $0\le t\le \tau^N(X)$ depending on the choice of $r$ term \eqref{term b} can be estimated by
\begin{align*}
 & \int_{0}^t\frac{1}{N}\sum_{\substack{j\neq i\\ j\in \mathcal{M}^N_b(X)}}\Big(\big|f^{N}([\Psi^{1,N}_{s,0}(X)]_j-[\Psi^{1,N}_{s,0}(X)]_i)\notag  \\
&- f^{N}(\varphi^{1,N}_{s,0}(X_j)-{\varphi^{1,N}_{s,0}}(X_i))\big|\Big) \mathbf{1}_{\widetilde{M}^N_{(r,R),(v,V)}(X_i)}(X_j) ds \\
\le & \int_{0}^t\frac{1}{N}\sum_{\substack{j\neq i\\ j\in \mathcal{M}^N_b(X)}} \Big(\big|f^{N}([\Psi^{1,N}_{s,0}(X)]_j-[\Psi^{1,N}_{s,0}(X)]_i)\big|\notag  \\
&+ \big|f^{N}(\varphi^{1,N}_{s,0}(X_j)-{\varphi^{1,N}_{s,0}}(X_i))\big|\Big)\mathbf{1}_{\widetilde{M}^N_{(r,R),(v,V)}(X_i)}(X_j) ds \mathbf{1}_{[0,4\delta^N_b]}(r)\notag \\
& +\frac{2}{N}\Delta^N_b(t,X)\sup_{Y\in \widetilde{M}^N_{(r,R),(v,V)}(X_i)}\int_0^t g^{N}(\varphi^{1,N}_{s,0}(Y)-{\varphi^{1,N}_{s,0}}(X_i))ds  \notag \\ 
& \cdot  \sum_{\substack{j\neq i\\ j\in \mathcal{M}^N_b(X)}} \mathbf{1}_{\widetilde{M}^N_{(r,R),(v,V)}(X_i)}(X_j)\mathbf{1}_{[4\delta^N_b,\infty)}(r) 
\end{align*}
where we applied Lemma~\ref{LipschitzLemma for f}, under the assumption that $\max\big(2N^{-\beta}, \frac{2}{3}|q|\big) \geq |\delta|$ holds.
The index $i$ is assigned to good particles, whereas $j$ is assigned to bad particles.
 Applying Corollary \ref{corollary phi and psi} we obtain that term \ref{term b} is bounded by
\begin{align}
\frac{C}{N}\frac{1}{\Delta v}\sum_{\substack{j\neq i\\ j\in \mathcal{M}^N_b(X)}} \mathbf{1}_{\widetilde{M}^N_{(r,R),(v,V)}(X_i)}(X_j)\mathbf{1}_{[0,4\delta^N_b]}(r)
+ \frac{C}{N}\frac{\Delta_b^N(t,X)}{\max(r,v)}\sum_{\substack{j\neq i\\ j\in \mathcal{M}^N_b(X)}} \mathbf{1}_{\widetilde{M}^N_{(r,R),(v,V)}(X_i)}(X_j)\mathbf{1}_{[4\delta^N_b,\infty)}(r) . \label{bound 2b}
\end{align}

We remark that the assumptions of the Corollary are indeed fulfilled in the current situation since according to the constraints on the possible parameters (see \eqref{cond M delta bad}) $r\in [0,4\delta^N_b]$ implies $R=\delta^N_b$ and $r=0$ an by regarding the definition of $ G^N(X_i)$ it follows that
$$\widetilde{M}^N_{(0,4\delta^N_b),(v,V)}(X_i)=M^N_{(0,4\delta^N_b),(v,V)}(X_i)\cap G^N(X_i)\subseteq \big(M^N_{4N^{-b_r},N^{-b_v}}(X_i)\big)^c$$  To derive an upper bound for term~\eqref{bound 2b} under the assumption that
\begin{align*}
\sum_{j \in \mathcal{M}b^N(X)}\mathbf{1}_{M^N_{(r,R),(v,V)}(X_i)}(X_j) \le N^{\gamma\sigma} \left\lceil N^{\gamma} R \min\big(\max(V, R), 1\big)^3 \right\rceil,
\end{align*}
we treat separately the contributions corresponding to the cases $\mathbf{1}_{[0, 4\delta^N_b]}(r)$ and $\mathbf{1}_{[4\delta^N_b, \infty)}(r)$. For the first of them we already discussed that $r=0$ and $R=4\delta^N_b$ due to condition \eqref{cond M delta bad}. We obtain that
\begin{align}
&\frac{C}{N}\frac{1}{\Delta v}\sum_{\substack{j\neq i\\ j\in \mathcal{M}^N_{b}(X)}}\mathbf{1}_{\widetilde{M}^N_{(r,R),(v,V)}(X_i)}(X_j)\mathbf{1}_{[0,4\delta^N_{b}]}(r)\notag\\
&\le\frac{C}{N}\left(\frac{R\min(V,1)^3|\mathcal{M}_{b}|}{\max( v,v_b)}+\frac{CN^{\gamma\sigma}}{\max( v,v_b)}\right)\notag\\
&\le\frac{CR \min(V,1)^3|\mathcal{M}_{b}|}{\max(N^{-b_v},v)N}+\frac{CN^{\gamma\sigma}}{N\max(N^{-b_v},v)}\notag\\
&\le C N^{-\frac{2}{5}-2\sigma-1+2-3b_v-b_r}\frac{ \min(V,1)^3}{\max(N^{-v_{b}},v)}+N^{b_v-1+\gamma\sigma}\notag \\
&\le C N^{-\frac{9}{5}-6\sigma}\frac{ \min(V,1)^3}{\max(N^{-v_{b}},v)}+N^{-\frac{2}{5}+2\sigma}  \label{term b.1}
\end{align} 
where we applied $R=\delta^N_b=N^{-\frac{2}{5}-2\sigma}$.\\
Taking additionally into account that $\Delta^N_b(t,X)\le N^{-\frac{2}{5}+2\sigma}$ as well as $R=N^{\sigma}r$ for $r\geq 4\delta^N_b$ (see \eqref{cond M delta bad}) it follows for the second term of \eqref{bound 2b} that
\begin{align}
& \frac{C}{N}\frac{\Delta_{b}^N(t,X)}{r\max(r,v)}\sum_{\substack{j\neq i\\ j\in \mathcal{M}^N_b(X)}} \mathbf{1}_{\widetilde{M}^N_{(r,R),(v,V)}(X_i)}(X_j) \notag \\
\le & \frac{C}{N}\Big( \frac{N^{-b_{r}-3b_{v}+2}R\min\big(\max(V,R),1\big)^3}{ r\max(r,v)}  +\frac{N^{\gamma\sigma}}{ \max(r,v)}\Big) N^{-\frac{2}{5}+2\sigma} \notag \\
\le & C\Big( \frac{\min\big(\max(V,R),1\big)^3}{\max(r,v)} N^{-\frac{9}{5}-\sigma}+N^{-\frac{7}{5}+\sigma}\Big).\label{term b.2}
\end{align}
In total term \eqref{term b} is bounded by the sum of term \eqref{term b.1} and term \eqref{term b.2}. 
All sets which belong to the family $\big(M^N_{(r_i,R_i),(v_i,V_i)}(Y)\big)_{i \in I_\sigma}$ are contained in a collision class, which takes one of the subsequent forms, except for $M^N_{(N^{-\sigma},\infty),(0,\infty)}(Y)$, which we will discuss separately:
\begin{multicols}{2}
\begin{itemize}
\item[(i)] $ M^N_{(0,4\delta^N_b),(0,4\delta^N_b)}(Y)$
\item[(ii)]$M^N_{(0,4\delta^N_b),(v,N^{\sigma}v)}(Y) $
\item[(iii)]  $M^N_{(0,4\delta^N_b),(1,\infty)}(Y)$
\item[(iv)] $M^N_{(r,N^{\sigma}r),(0,4\delta^N_b)}(Y)$
 \item[(v)] $M^N_{(r,N^\sigma r),(v,N^\sigma v)}(Y)$
\item[(vi)] $  M^N_{(r,N^\sigma r),(1,\infty)}(Y),$
\end{itemize}
\end{multicols}
\noindent for $r,v\in [0,1]$.
For all possible $r,R,v,V$ of the list term \eqref{term b.1} and term \eqref{term b.2} is bounded if $X\in \big(\mathcal{B}^{N,\sigma}_{3,i}\big)^c$ and $\sigma>0$ is chosen small enough and thus in total term \eqref{term b} is bounded by
\begin{align}
 &    CN^{-\frac{2}{5}+2\sigma},\label{bound2gb}
\end{align} 
as the number of collision classes $|I_{\sigma}|$ is independent of $N$.
The worst-case configuration among the possible sets arises when 
$v = N^{-\sigma}$ or $r = N^{-\sigma}$, with the sets being of type~(ii), (iv), or~(v).
For the last class, given by $M^N_{(N^{-\sigma},\infty),(0,\infty)}(Y)$, the following holds if $X\in \big(\mathcal{B}^{N,\sigma}_{3,i}\big)^c$ and times $t\le \tau^N(X)$:

\begin{align}
& \int_{0}^t\frac{1}{N}\sum_{\substack{j\neq i\\ j\in \mathcal{M}^N_{b}(X)}}\Big(\big|f^{N}([\Psi^{1,N}_{s,0}(X)]_j-[\Psi^{1,N}_{s,0}(X)]_i)\notag  - f^{N}(\varphi^{1,N}_{s,0}(X_j)-{\varphi^{1,N}_{s,0}}(X_i))\big|\Big) \mathbf{1}_{M^N_{(N^{-\sigma},\infty),(0,\infty)}(X_i)}(X_j) ds \notag\\
\le & \frac{2}{N}\sup_{Y\in M^N_{(N^{-\sigma},\infty),(0,\infty)}(X_i)}\int_0^t g^{N}(\varphi^{1,N}_{s,0}(Y)-{\varphi^{1,N}_{s,0}}(X_i))ds  \notag   \sum_{\substack{j\neq i\\ j\in \mathcal{M}^N_{b}(X)}} \underbrace{\Delta^N_b(t,X)}_{\le N^{-b\delta}}\notag \\
\le & \frac{2}{N} \big(T\frac{C}{(N^{-\sigma})^2}\big)N^{-b_{\delta}}\underbrace{|\mathcal{M}^N_{b}(X)|}_{\le  N^{2-b_r-3b_v}} \notag 
\le  CN^{-\frac{9}{5}}.
\end{align}
Now it only remains to show that the probability related to the set $\mathcal{B}_{3,i}^{N,\sigma}$ is indeed small enough. The proof proceeds as in \cite{grass}, with the argument adapted to the two-dimensional case.
It holds for $R,V>0$ 
\begin{align}
& \sum_{j \in \mathcal{M}_b^N(X)}\mathbf{1}_{M^N_{R,V}(X_i)}(X_j)\geq  \big\lceil N^{\frac{\sigma}{2}}\lceil N^{2-b_r-3b_v} R\min\big(\max(V,R),1\big)^3\rceil\big\rceil \label{unlikely set}\\
\Rightarrow & \big(\ \exists j \in \{1,...,N\}: \sum_{k=1}^N \mathbf{1}_{(G^N(X_j))^c}(X_k)\geq\lceil  \frac{N^{\frac{\sigma}{2}}}{2}\rceil  \big)\ \vee \label{unlikely set 1} \\
& \Big(\exists \mathcal{S}\subseteq \{1,...,N\}^2\setminus \bigcup_{n=1}^N\{(n,n)\}: \notag \\
& \ \text{(i)}\ \ \ |\mathcal{S}|= \left\lceil  \frac{N^{-\frac{\sigma}{2}}\big\lceil N^{\frac{\sigma}{2}}\lceil N^{2-b_r-3b_v} R\min\big(\max(V,R),1\big)^3\rceil\big\rceil}{2}\right\rceil  \notag \\ & \  \text{(ii)} \ \ \forall (j,k)\in \mathcal{S}:X_j\in(G^N(X_k))^c\cap M^N_{R,V}(X_i)   \notag \\
&\ \text{(iii)} \ (j_1,k_1),(j_2,k_2)\in \mathcal{S}\Rightarrow  \{j_1,k_1\}\cap \{j_2,k_2\}=\emptyset\Big),\label{unlikely set 2}
\end{align}
since $j\in \mathcal{M}^N_b(X)$ implies that there exists a $X_k\in \big(G^N(X_j)\big)^c$ for $k\in \{1,...,N\}\setminus \{j\}$.
For brevity of notation, we will refer to events of the form $X_m \in M^N_{R,V}(X_n)$ as collisions between particles $m$ and $n$. Similarly, we will refer to events of the form $X_m \in \big( G(X_n) \big)^c$ as hard collisions, 
although these are not collisions in the physical sense.

If the first assumption \eqref{unlikely set 1} does not hold, then an arbitrary bad particle can have at most $\left\lceil \frac{N^{\frac{\sigma}{2}}}{2}\right\rceil $ hard collisions with different particles.
Having a collision, and thus being a bad particle, is symmetric. 
In other words, it can infect at most 
$\left\lceil \frac{N^{\frac{\sigma}{2}}}{2} \right\rceil$ 
further particles to belong to the set $\mathcal{M}_b^N(X)$. Under this constraint the relation $$\sum_{j \in \mathcal{M}_b^N(X)}\mathbf{1}_{M^N_{R,V}(X_i)}(X_j)\geq \big\lceil N^{\frac{1}{2}\sigma}\lceil N^{2-b_r-3b_v} R\min\big(\max(V,R),1\big)\rceil \big\rceil $$ implies that the event related to \eqref{unlikely set 2} is fulfilled because if \eqref{unlikely set} is fulfilled, then there exists a set $\mathcal{C}_0\subseteq \mathcal{M}^N_b(X)$ with $ |\mathcal{C}_0|\geq \big\lceil N^{\frac{\sigma}{2}}\lceil N^{2-b_r-3b_v} R\min\big(\max(V,R),1\big)^3\rceil\big\rceil$ of bad particles which all have a collision with the particle labelled by $i$.
In the case that the event related to \eqref{unlikely set 1} does not occur there exist at most $\left\lfloor \frac{N^{\frac{\sigma}{2}}}{2}\right\rfloor $ particles having a hard collision with particle $i$.

 We construct a subset $\mathcal{C}_1\subseteq \mathcal{C}_0$ by detaching all of those, having a hard collision with particle $i$, which are (possibly) contained in $\mathcal{C}_0$ and then it holds that $$|\mathcal{C}_1|\geq \big\lceil N^{\frac{\sigma}{2}}\lceil N^{2-b_r-3b_v} R\min\big(\max(V,R),1\big)^3\rceil\big\rceil -\lfloor \frac{N^{\frac{\sigma}{2}}}{2}\rfloor \geq 1. $$Similarly we take one of these remaining bad particles $j_1$ out of $\mathcal{C}_1$ and since $j_1\in \mathcal{C}_1\subseteq \mathcal{C}_0\subseteq \mathcal{M}^N_b(X)$, there must be at least one further particle having a \textit{hard collisions} with $j_1$. By construction of $\mathcal{C}_1$ this can not be $i$. 
Lets pick one of this hard collision partners $k_1$ to get our first tuple $(j_1,k_1)$ which full fills condition (ii) of the set $\mathcal{S}$ appearing in \eqref{unlikely set 2}. 
In a next step detach $j_1$ and $k_1$ and all of their at most $(2\lfloor \frac{N^{\frac{\sigma}{2}}}{2}\rfloor -2)$ (possibly existing) remaining `\textit{hard collision} partners' from $\mathcal{C}_1$ to obtain a new set $\mathcal{C}_2\subseteq \mathcal{C}_1$. 
Finally, this gives us an iteration process (provided that $\mathcal{C}_2\neq \emptyset$) by choosing the next particle $j_2$ out of $\mathcal{C}_2$ and afterwards an arbitrary one of its hard collision partners $k_2$. Then the next round can start after having removed $j_2$ and $k_2$ as well as their (possibly existing) remaining hard collision partners from $\mathcal{C}_2$ to obtain $\mathcal{C}_3\subseteq \mathcal{C}_2$. 
By construction after each round of this process at most $2\lfloor \frac{N^{\frac{\sigma}{2}}}{2}\rfloor $ particle labels are removed from the set $\mathcal{C}_k$ to obtain $\mathcal{C}_{k+1}$.
 We repeat this procedure at least $$ \left\lceil \frac{\big\lceil N^{\frac{\sigma}{2}}\lceil N^{2-b_r-3b_v} R\min\big(\max(V,R),1\big)^3\rceil\big\rceil -\lfloor \frac{N^{\frac{\sigma}{2}}}{2}\rfloor }{N^{\frac{\sigma}{2}}}\right\rceil\geq \left\lceil \frac{N^{-\frac{\sigma}{2}}M}{2}\right\rceil $$ times and we additionally regarded that $\big\lceil N^{\frac{\sigma}{2}}\lceil N^{2-b_r-3b_v} R\min\big(\max(V,R),1\big)^3\rceil\big\rceil\geq N^{\frac{\sigma}{2}}$.
This provides us a set $\mathcal{S}$ consisting of tuples $(j_i,k_i)$ like claimed in \eqref{unlikely set 2}. The  removal of the hard collision partners of the occurring tuples after each round ensures that condition (iii) is fulfilled.  \\
Due to this considerations we can determine an upper bound for the probability $\mathbb{P}(X\in \mathcal{B}_{3,i}^{N,\sigma})$. Let to this end be $R, V>0$ and start with assumption \eqref{unlikely set 2} and abbreviate 
$$M_1:= \left\lceil \frac{N^{-\frac{\sigma}{2}}\big\lceil N^{\frac{\sigma}{2}}\lceil N^{2-b_r-3b_v} R\min\big(\max(V,R),1\big)^3\rceil\big\rceil}{2}\right\rceil .$$ 
There exist less than $\binom{N^2}{K}$ different possibilities to choose $K$ disjoint  (condition (iii) of \eqref{unlikely set 2} is fulfilled) pairs $(j,k)$ belonging to $\{1,...,N\}^2\setminus \bigcup_{n=1}^N\{(n,n)\}$. Application of this in the first step and Lemma \ref{Prob of group} and $\sup_{Y\in \mathbb{R}^4}\mathbb{P}\big(X_1\in (G^N(Y))^c\big)\le C N^{2-b_r-3b_v}$ yields that
{\allowdisplaybreaks \begin{align}
&\mathbb{P}\Big(\exists \mathcal{S}\subseteq \{1,...,N\}^2\setminus \bigcup_{n=1}^N\{(n,n)\}:|\mathcal{S}|=M_1 \ \land \notag \\ & \hspace{0,6cm} \big(\forall (j,k)\in \mathcal{S}:X_j\in(G^N(X_k))^c\cap M^N_{R,V}(X_i)  \big) \ \land \notag \\
& \hspace{0,6cm} \big((j_1,k_1),(j_2,k_2)\in \mathcal{S}\Rightarrow  \{j_1,k_1\}\cap \{j_2,k_2\}=\emptyset\big)\Big) \notag  \\
\le & \binom{N^2}{M_1}\mathbb{P}\Big( \forall (j,k)\in \{(2,3),(4,5),...,(2M_1,2M_1+1)\}: \notag \\
&\hspace{1,5cm} X_j\in(G^N(X_k))^c\cap M^N_{R,V}(X_1) \Big) \notag  \\
\le &\frac{N^{2M_1}}{M_1!}\Big(\sup_{Y\in \mathbb{R}^4}\mathbb{P}\big(X\in (G^N(Y))^c\big)
\sup_{Z\in \mathbb{R}^4}\mathbb{P}\big(X\in M^N_{R,V}(Z)\big)\Big)^{M_1} \notag \\
\le & C^{M_1}\frac{N^{2M_1}}{M_1^{M_1}} \big(N^{-b_r-3b_v}\big)^{M_1}\Big(R\min\big(\max(V,R),1\big)^3\Big)^{M_1}\notag \\
\le &  (CN^{-\frac{5\sigma}{2}} )^{\frac{N^{\frac{\sigma}{2}} }{2}},
\end{align} }
where we used in the last step, that $M_1\geq \frac{N^{\frac{\sigma}{2}}}{2} $. 
For any class $\big(M^N_{(r_i,R_i),(v_i,V_i)}(Y)\big)_{i \in I_{\delta}}$ with $R_l< \infty$ this probability decays distinctly faster than necessary.
 
By setting the collision class parameters $R,V$ to infinity to obtain the event $\mathbf{1}_{M^N_{\infty,\infty}(X_i)}(X_j)=1$ and alter redefining $M_1:= \lceil \frac{N^{2-b_r-3b_v+\frac{\sigma}{2}}}{2}\rceil$ and thus $\mathbb{P}\big(X_1\in M^N_{R,V}(Y)\big)=1$ we can apply the considerations from above to show that $\sum_{k\in \mathcal{M}^N_b(X)}1\le N^{2-b_r-3b_v(1+\sigma)}$. Thus we get
$$\mathbb{P}\big(\sum_{k\in \mathcal{M}^N_b(X)}1\le N^{2-b_r-3b_v}\big)\le CN^{-\sigma N^{2-b_r-3b_v}} $$ 
which is small enough.\\ 
To prove assumption \eqref{unlikely set 1} we abbreviate $M_2:= \lceil  \frac{N^{\frac{\sigma}{2}}}{2}\rceil $ and it holds that
\begin{align}
& \mathbb{P}\Big(X\in \mathbb{R}^{4N}: \big(\exists j \in \{1,...,N\}: \sum_{k\neq j} \mathbf{1}_{(G^N(X_j))^c}(X_k)\geq  M_2\big)\Big)\notag \\
\le & N\mathbb{P}\Big(X\in \mathbb{R}^{4N}: \sum_{k=2}^N \mathbf{1}_{(G^N(X_1))^c}(X_k)\geq M_2\Big) \notag \\
\le & N\binom{N}{M_2}\sup_{Y\in \mathbb{R}^4}\mathbb{P}\big(Z\in \mathbb{R}^{6}: Z\in (G^N(Y))^c  \big)^{M_2} \notag \\
\le & N \frac{N^{M_2}}{M_2!}  \big(CN^{-b_r-3b_v}\big)^{M_2} \notag \\
\le & CN^{1-b_r-3b_v\lceil  \frac{N^{\frac{\sigma}{2}}}{2}\rceil }. \label{prob.est.} 
\end{align}
In total we obtain
\begin{align}
& \mathbb{P}\big(X\in \mathcal{B}^{N,\sigma}_{3,i} \big) \notag \\
\le & |I_{\sigma}|\sup_{\substack{R,V>0}}\mathbb{P}\Big(\sum_{j \in \mathcal{M}_b^N(X)}\mathbf{1}_{M^N_{R,V}(X_i)}(X_j) \geq  N^{\frac{1\sigma}{2}}\big\lceil N^{2-b_r-b_v} R\min\big(\max(R,V),1\big)^3\big\rceil \Big) \notag \\
& + \mathbb{P}\big(\sum_{k\in \mathcal{M}^N_b(X)}1\geq N^{2-b_r-b_v}\big)  \notag\\
\le &(CN^{-\frac{5\sigma}{2}} )^{\frac{N^{\frac{\sigma}{2}}}{2} } 
 \label{prob.b.3b}
\end{align}
\subsubsection{Estimate of Term \eqref{term c}} 
Now we are left with the last term (\ref{term c}) which measures the fluctuation between two good particles. 
To estimate the term we identify $ v^N_{min}=N^{-b_v}=N^{-\frac{3}{5}-\sigma}.$
Analogous to the estimates for the previous terms we apply Corollary \ref{corollary phi and  psi} and additionally subdivide the term depending on the relative velocity of the particles.
\begin{align}
& \int_{0}^t\Big(\frac{1}{N}\sum_{\substack{j\neq i\\ j\in \mathcal{M}^N_g(X)}}\Big(\big|f^{N}([\Psi^{1,N}_{s,0}(X)]_j-[\Psi^{1,N}_{s,0}(X)]_i)\big| \notag  \\
&+ \big|f^{N}(\varphi^{1,N}_{s,0}(X_j)-{\varphi^{1,N}_{s,0}}(X_i))\big|\Big)\mathbf{1}_{G^N(X_i)\cap M^N_{2N^{-\frac{2}{5}+2\sigma},\infty}(X_i)}(X_j)\Big) ds \notag \\
\le & \frac{C}{N}\frac{1}{ v^N_{min}} \sum_{j\neq i} \mathbf{1}_{G^N(X_i)\cap M^N_{2N^{-\frac{2}{5}+\sigma},N^{-\frac{1}{5}-\sigma}}(X_i)}(X_j) \notag \\
&+  \frac{C}{N}\frac{1}{N^{-\frac{1}{5}-\sigma}} \sum_{j\neq i} \mathbf{1}_{ G^N(X_i)\cap  M^N_{2N^{-\frac{2}{5}+2\sigma},\infty}(X_i)}(X_j). \label{term c estimated}
\end{align}

Corollary \ref{corollary phi and  psi} (ii) is applicable since the relative velocity values for the considered `collision classes' are of distinctly larger order than the deviation between corresponding particle trajectories of the microscopic and the auxiliary system. More precisely, we applied that $G^N(X_i)\subseteq M(^N_{ 4\delta^N_b, v^N_{min}}(X_i))^c$ where $\delta_b^N=N^{-\frac{2}{5}-2\sigma}$ if
\begin{align*}
&\max_{i\in \mathcal{M}^N_g(X)}\sup_{0\le s \le \tau^N(X)}|[\Psi^N_{s,0}(X)]_i-\varphi^N_{s,0}(X_i)|
\le 
N^{-\frac{2}{5}+2\sigma} 
\end{align*}
 This shows that the assumptions of Corollary \ref{corollary phi and  psi} (ii) are fulfilled .\\
Now we define a fourth set of `inappropriate' initial data as follows:
\begin{align*}
\begin{split}
& X\in \mathcal{B}_{4,i}^{N,\sigma}\subseteq \mathbb{R}^{4N} 
\Leftrightarrow   \sum_{j \neq i}\mathbf{1}_{M^N_{4N^{-\frac{3}{5}+2\sigma},N^{-\frac{1}{5}-\sigma}}(X_i)}(X_j)\geq  N^{\sigma}\ \land  \sum_{j \neq i}\mathbf{1}_{M^N_{4N^{-\frac{2}{5}+2\sigma},\infty}(X_i)}(X_j)\geq  N^{\frac{1}{5}+3\sigma}
\end{split} 
\end{align*}
Due to our estimates it holds for configurations belonging to the complement of this set that term \eqref{term c estimated} (and thereby \eqref{term c}) is bounded by
\begin{align*}
\frac{c}{v^N_{b}} N^{\sigma-1} + \frac{C}{N^{1-\frac{1}{5}-\sigma}} N^{\frac{1}{5}+3\sigma}
\le  
CN^{-\frac{2}{5}+2\sigma} +CN^{-1+4\sigma} \le CN^{-\frac{2}{5}+2\sigma}
\end{align*}

We shorten this time $M_1:= \lceil N^{\sigma} \rceil$, $M_2:= \lceil N^{\frac{1}{5}+3\sigma} \rceil$ and by the same justification as applied in \eqref{prob.est.} and application of Lemma \ref{Prob of group} we can estimate the probability
\begin{align*}
& \mathbb{P}\left(X\in  \mathcal{B}^{N,\sigma}_{4,i}\right) \\
\le & \frac{N^{M_1}}{M_1!}\sup_{Y\in \mathbb{R}^4}\mathbb{P}\big(X_i\in M^N_{4N^{-\frac{2}{5}+2\sigma},N^{-\frac{1}{5}-\sigma}}(Y)\big)^{M_1}   + \frac{N^{M_2}}{M_2!}\sup_{Y\in \mathbb{R}^4}\mathbb{P}\big(X_i\in M^N_{4N^{-\frac{2}{5}+2\sigma},\infty}(Y)\big)^{M_2} \\
\le & \frac{(CN)^{M_1}}{M_1!}\big(N^{(-\frac{2}{5}-2\sigma)}\big)^{M_1}\big(N^{3(-\frac{1}{5}-\sigma)}\big)^{M_1}  + C^{M_2}\frac{N^{M_2}}{(N^{\frac{1}{5}+3\sigma})^{M_2}}\big(N^{
2(-\frac{2}{5}+2\sigma)}\big)^{M_2}   \\
\le &C\big(N^{-\sigma} \big)^{N^{\sigma}} +\big( CN^{-\sigma}\big)^{N^{\frac{1}{5}+3
\sigma}} 
\end{align*} 
which for $\sigma>0$ decreases fast enough.
\subsection{Conclusion}
Due to the previous probability estimates on the unlikely sets it easily follows that for small enough $\sigma>0$ and an arbitrary $\gamma>0$ there exists a constant $C>0$ such that $$\mathbb{P}\big(\bigcup_{j\in \{1,2,3,4\}}\bigcup_{i=1}^N \mathcal{B}^{N,\sigma}_{j,i}\big)\le CN^{-\gamma}.$$ 
Further we can show, that in the two-dimensional setting, a bad particle typically does not occur. Therefore we define a fifth set
\begin{align*}
 X \in \mathcal{B}_5^{N,\sigma}
\Leftrightarrow\  X \in \mathbb{R}^{4N} : \big(\exists (i,j) \in \{1, \dots, N\}^2,\, i \neq j \text{ and } X_j \notin G^N(X_i)\big). 
\end{align*}
Its probability is bounded by
\begin{align*}
\mathbb{P}\big(X \in \mathcal{B}_5^{N,\sigma} \big)
\le \binom{N}{2} \, \mathbb{P}\big(X_2 \notin G^N(X_1)\big)
\le C N^2 \cdot N^{-b_r - 3b_v}
= C N^{-\frac{2}{5}-4\sigma}.
\end{align*}
Thus we further restrict the initial data on
\begin{align*}
 \mathcal{G}^{N,\sigma}_1:= \Big(\mathcal{B}^{N,\sigma}_5\cup \bigcup_{j\in \{1,2,3,4\}}\bigcup_{i=1}^N \mathcal{B}^{N,\sigma}_{j,i}\Big)^c.  
\end{align*}
The probability of excluded configurations full fills
$$
\mathbb{P}\Big(X \in \big(\mathcal{G}_1^{N,\sigma}\big)^c\Big) \le C N^{-\frac{2}{5} - 4\sigma}.
$$
The parameter $\sigma > 0$ can be chosen arbitrarily small, and in particular, smaller than $\frac{\epsilon}{2}$ for any given $\epsilon > 0$.
It remains to show that for such configurations the stopping time does not get triggered. 
All terms are bounded by $CN^{-\frac{2}{5}+2\sigma}$ and we are still left with \eqref{term c2}.
\subsubsection{Estimate of Term \eqref{term c2} }
For $X\in\mathcal{G}_1^{N,\sigma}$ it holds for any $i\in \lbrace,...,N\rbrace$ and a points in time $t_1,t\in \left[0,T\right]$ that 
\begin{align*}
& \big|\int_{t_1}^t\Big(\frac{1}{N}\sum_{j=1}^N g^{N}(\varphi^{1,N}_{s,0}(X_j)-{\varphi^{1,N}_{s,0}}(X_i))\mathbf{1}_{G^N(X_i)}(X_j)\\
&  -\int_{\mathbb{R}^4} g^{N}(\varphi^{1,N}_{s,0}(Y)-{\varphi^{1,N}_{s,0}}(X_i))\mathbf{1}_{G^N(X_i)}(Y)k_0(Y)
d^4Y\Big)ds\Big|\le 1
\end{align*}
by the definition of $\mathcal{B}^{N,\sigma}_{2,i}$
and thus for $N>1$ and $t_1\le t$ we receive 
\begin{align}
& \int_{t_1}^t\frac{1}{N}\sum_{j=1}^N g^{N}(\varphi^{1,N}_{s,0}(X_j)-{\varphi^{1,N}_{s,0}}(X_i))\mathbf{1}_{G^N(X_i)}(X_j)ds \notag \\
\le & 1+\int_{t_1}^t\int_{\mathbb{R}^4} g^{N}(\varphi^{1,N}_{s,0}(Y)-{\varphi^{1,N}_{s,0}}(X_i))\mathbf{1}_{G^N(X_i)}(Y)k_0(Y)d^4Yds. \notag \\
\le & 1+ C\sup_{t_1\le s\le t}\int_{\mathbb{R}^3}\min\big(N^{2\beta},\frac{1}{|Y-{\varphi^{1,N}_{s,0}}(X_i)|^2}\big)\widetilde{k}^N_s(Y)d^2Y \notag \\
\le &1+ C\ln(N)(t-t_1). \label{est.sum.g}
\end{align}
For times $t_1\le t$  it holds that 
\begin{align}\label{Delta_g}
\Delta^N_g(t,X)\le \Delta^N_g(t_1,X)+\int_{t_1}^t\delta^N_g(s,X)ds.
\end{align}
If we choose for some constant $C>0$ the subsequent sequence of time steps $t^*:= t_{n+1}-t_n= \frac{C}{\sqrt{\ln(N)}}$ with
\begin{align*}
& t_n=n\frac{C}{\sqrt{\ln(N)}} \text{ for }n\in \{0,...,\lceil \frac{\sqrt{\ln(N)}}{C}\tau^N(X)\rceil-1\},\\
& t_{\lceil \frac{\sqrt{\ln(N)}}{C}\tau^N(X)\rceil}=\tau^N(X)
\end{align*}
\eqref{Delta_g} implies that for $ t_n\le t\le \tau^N(X)$ 
\begin{align*}
\Delta^N_g(t,X)\le \sum_{k=1}^n\sup_{0\le s \le t_k}\delta^N_g(s,X)t^*+\int_{t_n}^t\delta^N_g(s,X)ds.
\end{align*}
Thus for any good particle $i\in \mathcal{M}^N_g(X)$ and all times $t\in [t_n,t_{n+1}]$ with $n\in \{0,...,\lceil \frac{\sqrt{\ln(N)}}{C}\tau^N(X)\rceil-1\}$ holds:
\begin{align}\label{delta estimate}
&\delta^N_g(t,X)\notag \\
\le & \delta^N_g(t_n,X)+\max_{i\in \mathcal{M}^N_g(X)} \big|\int_{t_n}^t\Big(\frac{1}{N}\sum_{j\neq i}f^{N}([\Psi^{1,N}_{s,0}(X)]_i-[\Psi^{1,N}_{s,0}(X)]_j) \notag \\
&   -f^{N}*\widetilde{k}^N_s({\varphi^{1,N}_{s,0}}(X_i))\Big)ds\big| \notag  \\
\le &\max_{i\in \{1,...,N\}}\int_{t_n}^t \frac{2}{N}\sum_{j=1}^N g^{N}(\varphi^{1,N}_{s,0}(X_j)-{\varphi^{1,N}_{s,0}}(X_i))\mathbf{1}_{G^N(X_i)}(X_j)\underbrace{\Delta^N_g(s,X)}_{\le \Delta^N_g(t,X)}ds  \notag \\
&+\delta^N_g(t_n,X) +CN^{-\frac{2}{5}+2\sigma} \notag \\
\le &\big(1+ C\ln(N)\underbrace{(t-t_n)}_{\le t_{n+1}-t_n=t^*}\big)\big( \sum_{k=1}^n\sup_{0\le s \le t_k}\delta^N_g(s,X)t^*+\int_{t_n}^t\delta^N_g(r,X)dr\big)\notag \\
& +\delta^N_g(t_n,X)+CN^{-\frac{2}{5}+2\sigma} \notag \\
\le &\big(1+ C\ln(N)t^*\big)\int_{t_n}^t\delta^N_g(r,X) dr\notag \\
&  + \big(2+ C\ln(N)(t^*)^2\big)\sum_{k=1}^n\sup_{0\le s \le t_k}\delta^N_g(s,X) +CN^{-\frac{2}{5}+2\sigma}. 
\end{align} 
For all times $t\in [t_n,t_{n+1}]$ Lemma \ref{Gronwall Lemma} implies that 
\begin{align*}
 \delta^N_g(t,X)
\le  \left(\left(2+ C\ln(N)(t^*)^2\right)\sum_{k=1}^n\sup_{0\le s \le t_k}\delta^N_g(s,X) +CN^{-\frac{2}{5}+2\sigma}\right) e^{t^*+ C\ln(N){(t^*)}^2}.
\end{align*}
For times $t\in [0,t_n]$, we can exchange $\delta^N_g(t,X)$ by its supremum over $[0,t_{n+1}]$ and thus for $t^*=\frac{C}{\sqrt{\ln(N)}}$ with $C:= \min\big(\frac{1}{\sqrt{C}},1\big)$ \eqref{delta estimate} implies 
\begin{align*}
 & \sup_{0\le s \le t_{n+1}}\delta^N_g(s,X)
\le   2e^2\sum_{k=1}^n\sup_{0\le s \le t_k}\delta^N_g(s,X) +Ce^2N^{-\frac{2}{5}+2\sigma}. 
\end{align*}
It follows that
\begin{align*}
& \sup_{0\le s \le t_n}\delta^N_g(s,X)
\le  Ce^2N^{-\frac{2}{5}+2\sigma}(2e^2+1)^{n-1}, 
\end{align*}
for $n\in \lbrace,...,\lceil \frac{\sqrt{\ln(N)}}{C}\tau^N(X)\rceil\rbrace$.
If this relation is valid for $k\in \{1,...,n\}$, $n\in \N$ then it holds that
\begin{align*}
& \sup_{0\le s \le t_{n+1}}\delta^N_g(s,X)\\
\le &  2e^2\sum_{k=1}^n \underbrace{\sup_{0\le s \le t_k}\delta^N_g(s,X) }_{\le   Ce^2N^{-\frac{2}{3}+2\sigma}(2e^2+1)^{k-1}}+Ce^2N^{-\frac{2}{5}+2\sigma} \\
\le & 2e^2\big(  Ce^2N^{-\frac{2}{5}+2\sigma}\frac{(2e^2+1)^n-1}{(3e^2+1)-1}\big)+Ce^2N^{-\frac{2}{5}+2\sigma}\\
=&  Ce^2N^{-\frac{2}{5}+2\sigma}(2e^2+1)^n.
\end{align*}
For $N\in \mathbb{N}$ large enough the velocity deviation is bounded by
\begin{align*}
\sup_{0\le s \le \tau^N(X)}\delta^N_g(s,X)
\le &  Ce^2N^{-\frac{2}{5}+2\sigma}(2e^2+1)^{\lceil \frac{\sqrt{\ln(N)}}{C}\tau^N(X)\rceil-1} \notag \\
\le & Ce^2N^{-\frac{2}{5}+2\sigma}N^{ \frac{\ln(2e^2+1)}{\ln(N)}\frac{\sqrt{\ln(N)}}{C}T} \notag \\
\le & CN^{-\frac{2}{5}+2\sigma}
\end{align*}
 and this implies
\begin{align*}
& \max_{i\in \mathcal{M}^N_g(X)}\sup_{0\le s \le \tau^N(X)}\left|[\Psi^N_{s,0}(X)]_i-\varphi^N_{s,0}(X_i)\right|  
\le  CN^{-\frac{2}{5}+2\sigma}.
\end{align*} 

\subsection{Propagation of chaos}
We conclude the proof of Theorem \ref{maintheorem} by showing that for $N>1$
\begin{align}
\sup\limits_{x\in \mathbb{R}^4}\sup_{0\le s \le T}|\varphi_{s,0}^{1,N}(x)-{\varphi^{1,\infty}_{s,0}}(x)|\le \frac{e^{C\sqrt{\ln(N)}}}{N^{4}}. \label{Distance phi, phiinf}
\end{align}
For simplicity we denote $\Delta_N(t):= \sup\limits_{x\in \mathbb{R}^4}\sup_{0\le s \le t}|\varphi_{s,0}^{1,N}(x)-{\varphi^{1,\infty}_{s,0}}(x)|$. 
For $x\in \mathbb{R}^4$ and $N\in \N\setminus \{1\}$ and a point in time $t\in [0,T]$, such that $\Delta_N(t)\le N^{-2}$, it holds that
\begin{align*}
 & |\varphi_{t,0}^{2,N}(x)-{\varphi^{2,\infty}_{t,0}}(x)|\\
\le &\big|\int_0^t\int_{\mathbb{R}^4}\big(f^N(\varphi^{1,N}_{s,0}(x)-{\varphi^{1,N}_{s,0}}(y))-f^{\infty}({\varphi^{1,\infty}_{s,0}}(x)-{\varphi^{1,\infty}_{s,0}}(y))\big)k_0(y)d^4Yds\big|\\
\le &\big|\int_0^t\int_{\mathbb{R}^4}\big(f^N(\varphi^{1,N}_{s,0}(x)-{\varphi^{1,N}_{s,0}}(y))-f^N({\varphi^{1,\infty}_{s,0}}(x)-{\varphi^{1,\infty}_{s,0}}(y))\big)k_0(y)d^4Yds\big|\\
 & +\big|\int_0^t\int_{\mathbb{R}^4}\big(f^N(\varphi^{1,\infty}_{s,0}(x)-{\varphi^{1,\infty}_{s,0}}(y))-f^{\infty}({\varphi^{1,\infty}_{s,0}}(x)-{\varphi^{1,\infty}_{s,0}}(y))\big)k_0(y)d^4Yds\big|\\
\le & \int_{0}^t2\Delta_N(s)\int_{\mathbb{R}^4} g^N({\varphi^{1,N}_{s,0}}(x)-{\varphi^{1,N}_{s,0}}(y))k_0(y)d^4Yds\\
 & +\big|\int_0^t\int_{\mathbb{R}^4}\big(f^N({\varphi^{1,\infty}_{s,0}}(x)-{^1y})-f^{\infty}({\varphi^{1,\infty}_{s,0}}(x)-{^1y})\big)k^{\infty}_s(y)d^4Yds\big|\\
\le & C\ln(N)\int_{0}^t\Delta(s) ds+
\big|\int_0^t\int_{\mathbb{R}^4} \frac{^1y}{|^1y|^{2}}\mathds 1_{(0,N^{-2}]}(|^1y|) k^{\infty}_s(y+{\varphi}^{\infty}_{s,0}(x))d^4Yds\big|\\
&+ \big|\int_0^t\int_{\mathbb{R}^4} {^1y} N^{2}\mathds 1_{(0,N^{-2}]}(|^1y|) k^{\infty}_s(y+{\varphi}^{\infty}_{s,0}(x))d^4Yds\big|
\end{align*}
since $g^N(q)\le C\min\big( N^{4},\frac{1}{|q|^{2}}\big)$ for all $q\in \mathbb{R}^2$.
Further
\begin{align*}
& \big|\int_0^t\int_{\mathbb{R}^4} \frac{^1y}{|^1y|^{2}}\mathds 1_{(0,N^{-2}]}(|^1y|)k^{\infty}_s((0,{^2y})+{\varphi_{s,0}^{\infty}}(x))d^4Yds\big|\\
=&\big|\int_0^t\widetilde{k}^{\infty}_s({^1\varphi_{s,0}^{\infty}}(x))\int_{\mathbb{R}^2} \frac{q}{|q|^{2}}\mathds 1_{(0,N^{-2}]}(|q|)d^2qds\big|= 0,
\end{align*}
and thus we can estimate
\begin{align*}
&\big|\int_0^t\int_{\mathbb{R}^4} \frac{^1y}{|^1y|^{2}}\mathds 1_{(0,N^{-2}]}(|^1y|) k^{\infty}_s(y+{\varphi}^{\infty}_{s,0}(x))d^4Yds\big|\notag \\
= &\big|\int_0^t\int_{\mathbb{R}^4} \frac{^1y}{|^1y|^{2}}\mathds 1_{(0,N^{-2}]}(|^1y|) \Big(\big(k^{\infty}_s(y+{\varphi}^{\infty}_{s,0}(x))-k^{\infty}_s((0,{^2y})+{\varphi}^{\infty}_{s,0}(x))\big) \notag \\
&+k^{\infty}_s((0,{^2y})+{\varphi}^{\infty}_{s,0}(x))\Big)d^4Yds\big| \notag \\
\le  & \int_0^t\int_{\mathbb{R}^4} \frac{1}{|^1y|}\mathds 1_{(0,N^{-2}]}(|^1y|)\Big(\big|k^{\infty}_s(y+{\varphi}^{\infty}_{s,0}(x))-k^{\infty}_s((0,{^2y})+{\varphi}^{\infty}_{s,0}(x))\big|\Big)d^4Yds. 
\end{align*}
Further $|\nabla k_0(x)|\le \frac{C}{(1+|x|)^{2+\delta}}$ and by application of Lemma \ref{Lemma distance same order} it follows for arbitrary $z\in \mathbb{R}^4$ and $s\in \left[0,T\right]$ that
\begin{align*}
&\big|k^{\infty}_s(y+z)-k^{\infty}_s((0,{^2y})+z)\big|\mathds 1_{(0,N^{-2}]}(|^1y|) \\
=& \big|k_0(\varphi^{\infty}_{0,s}(y+z))-k_0(\varphi^{\infty}_{0,s}((0,{^2y})+z))\big| \mathds 1_{(0,N^{-2}]}(|^1y|)   \\
\le & \sup_{z'\in \overline{\varphi^{\infty}_{0,s}(y+z)\varphi^{\infty}_{0,s}((0,{^2y})+z)}}|\nabla k_0(z')| \cdot \mathds 1_{(0,N^{-\beta}]}(|^1y|)\Big( \big|\varphi^{\infty}_{0,s}(y+z)-\varphi^{\infty}_{0,s}((0,{^2y})+z)\big|\Big) \notag \\
 \le &   \sup_{z'\in \overline{\varphi^{\infty}_{0,s}(y+z)\varphi^{\infty}_{0,s}((0,{^2y})+z)}}\frac{C}{(1+|z'|)^{2+\delta}}  \cdot \mathds 1_{(0,N^{-2}]}(|^1y|)\Big(C\big|(y+z)-\big((0,{^2y})+z \big) \big|\Big) \notag \\
\le & \sup_{y'\in \mathbb{R}^2:|y'|\le N^{-\beta}}\sup_{z'\in \overline{\varphi^{\infty}_{0,s}((y',{^2y})+z)\varphi^{\infty}_{0,s}((0,{^2y})+z)}}\frac{CN^{-2}}{(1+|z'|)^{2+\delta}} 
\end{align*}
where $\overline{xy}:= \{(1-\eta)x+\eta y\in \mathbb{R}^4: \eta\in [0,1]\}$ for $x,y\in \mathbb{R}^4$.
This term decreases like $\frac{CN^{-2}}{(1+|^2y|)^{2+\delta}}$ as $|^2y|$ increases.  
Thus it follows that for arbitrary $z\in \mathbb{R}^4$ 
 \begin{align*}
&\big|\int_{\mathbb{R}^4} \frac{^1y}{|^1y|^{2}}\mathds 1_{(0,N^{-2}]}(|^1y|) k^{\infty}_s(y+z)d^4Y\big|\\
\le  &\int_{\mathbb{R}^2} \frac{1}{|^1y|}\mathds 1_{(0,N^{-2}]}(|^1y|)d^2(^1y) \cdot \int_{\mathbb{R}^2} \sup_{y'\in \mathbb{R}^2:|y'|\le N^{-2}}\sup_{z'\in \overline{\varphi^{\infty}_{0,s}((y',{^2y})+z)\varphi^{\infty}_{0,s}((0,{^2y})+z)}}\frac{CN^{-\beta}}{(1+|z'|)^{2+\delta}}d^2(^2y)\\
\le & \frac{C}{N^{2}}.
\end{align*}
Further it holds that for any $x\in \mathbb{R}^4$ 
\begin{align*}
\sup_{0\le s \le t}|\varphi_{s,0}^{1,N}(x)-{\varphi^{1,\infty}_{s,0}}(x)|
\le  \int_0^t|\varphi_{s,0}^{2,N}(x)-{\varphi^{2,\infty}_{s,0}}(x)| ds
\le  C\ln(N)\int_{0}^t\int_0^s\Delta_N(r)drds+\frac{Ct}{N^{2}}. 
\end{align*}
By Lemma \ref{Gronwall Lemma} it follows that $$ \Delta_N(t)=\sup_{x\in \mathbb{R}^4}\sup_{0\le s \le t}|{\varphi_{s,0}^{1,N}}(x)-{\varphi^{1,\infty}_{s,0}}(x)|\le \frac{Cte^{\sqrt{C\ln(N)}t}}{N^{2}}.$$ 
For arbitrarily large times $t$ provided that $N\in \mathbb{
N}$ is large enough we can conclude $\Delta_N(t)\le N^{-2}$.
Applying this bound on $$|\varphi_{t,0}^{2,N}(x)-{\varphi^{2,\infty}_{t,0}}(x)| \notag 
\le C\ln(N)\int_{0}^t\Delta_N(s)ds+CN^{-2}$$ yields for $N$ large enough the claimed result 
\begin{align}
\sup\limits_{x\in \mathbb{R}^4}\sup_{0\le s \le T}|\varphi_{s,0}^N(x)-{\varphi^{\infty}_{s,0}}(x)|\le
\frac{ e^{C\sqrt{\ln(N)}}}{N^{2}}. 
\end{align}

\end{document}